\documentclass[12pt]{amsart}
\usepackage{amsmath}
\usepackage{amsthm}
\usepackage{color}
\usepackage{cite}
\usepackage{mathrsfs}
\usepackage{amssymb,amsmath,amsthm,color}
\usepackage{graphicx, cite}
\usepackage{hyperref}
\usepackage{url}
\usepackage{setspace}
\usepackage{enumerate}

\usepackage{tikz}
\usepackage{amssymb}
\usepackage{mathrsfs}
\usepackage{latexsym}
\theoremstyle{plain}
\newtheorem{thm}{Theorem}[section]
\newtheorem{lem}[thm]{Lemma}
\newtheorem{cor}[thm]{Corollary}
\newtheorem{prop}[thm]{Proposition}
\theoremstyle{definition}
\newtheorem{rem}[thm]{Remark}

\newtheorem{defn}[thm]{Definition}

\numberwithin{equation}{section}

\newcommand{\R}{\mathbb{R}}
\newcommand{\C}{\mathbb{C}}

\newcommand{\N}{\mathbb{N}}

\newcommand{\U}{U_{\alpha}(t)}
\newcommand{\pa}{\partial}

\newcommand{\Lp}{\widehat{L^p}}

\newcommand{\HSP}{\widehat{H^{s}_p}}
\newcommand{\HSCP}{\widehat{H^{s_c}_p}}
\newcommand{\HHSP}{\widehat{\dot{H}^{s}_p}}
\newcommand{\HHSCP}{\widehat{\dot{H}^{s_c}_p}}

\newcommand{\A}{\alpha}

\newcommand{\simleq}{
\hspace{0.3em}\raisebox{0.4ex}{$<$}\hspace{-0.75em}\raisebox{-.7ex}{$\sim$}\hspace{0.3em}}

\newcounter{kotaeflg}
\newcommand{\kotae}[1]{
\ifodd \arabic{kotaeflg}
#1
\fi
}

\begin{document}

\title[Cauchy problem for fractional NLS]%
      {Fefferman--Stein-type estimates and  fractional NLS\\ in Fourier Sobolev spaces}

\author[Divyang G. Bhimani]{Divyang G. Bhimani}
\address{Divyang G. Bhimani\\Department of Mathematics\\Indian Institute of Science Education and Research\\ Pune 411008\\India}
\email{divyang.bhimani@iiserpune.ac.in}
      
\author[Ryosuke Hyakuna]{Ryosuke Hyakuna}
\address[Polytechnic University of Japan]{}
\email{107r107r@gmail.com}
\date{\today}
\begin{abstract}
In this paper, we prove a sharp Fefferman--Stein-type estimate for the fractional Schr\"odinger equation, which can be regarded as a generalized Strichartz estimate for data in the Fourier Lebesgue space $\Lp$.  Then, as  an application of the Fefferman--Stein inequality and its off-diagonal generalization, we prove large data local well-posedness and small data global well-posedness results for the one dimensional fractional nonlinear Schr\"odinger equation with pure power nonlinearities in the homonegeneous and inhomogeneous Fourier--Sobolev spaces $\HHSP,\HSP$.  Solutions are established in $L_x^r(\R ;L^q_t(I))$ spaces in order to overcome the difficulty of a loss of derivatives in the standard Strichartz estimates.
\end{abstract}
\maketitle
\tableofcontents
\section{Introduction and Main Results}
\subsection{Fefferman--Stein-type Estimates}The Cauchy problem for the fractional nonlinear Schr\"odinger equation (NLS for short) in $\R^d,\,d \ge 1$ is
\begin{equation}
iu_t+(-\Delta)^{\A/2}u +N(u)=0,\quad u|_{t=0}=\phi, \label{gFNLS}
\end{equation}
where $\A \in (0,2]\setminus \{1\}$ and $N(u)$ is a nonlinearity. 

The fractional Schr\"odinger equations were
introduced in the theory of the fractional quantum mechanics where the Feynmann
path integral approach is generalized to the $\alpha$-stable L\'evy process \cite{LaskinFracLap}.  It also appears in the water wave models, see e.g. \cite{IonescuFracNLS}.

When $\A=2$, \eqref{gFNLS} is the well-known nonlinear Schr\"odinger equation.  Cauchy problem \eqref{gFNLS} has been extensively studied in recent years.  In particular, well-posedness results in the Sobolev space $H^s$ were established for $N(u)=(|x|^{-\gamma} \ast |u|^2)u$ in \cite{Cho_2013_frHartreeNLS}, for $N(u)=|u|^{\rho-1}u$ in \cite{HongSire_2015_frNLS}, and for the 1D cubic nonlinearity in \cite{Cho_2015_well_ill_NLS}.  As is well known, the Strichartz estimates are powerful tool to establish solutions to the nonlinear Schr\"odinger equations. Their generalization to the fractional equation is also a key to obtain well-posedness of \eqref{gFNLS} in  these earlier studies.  The following Strichartz  estimates for the fractional Schr\"odinger equation are established in \cite[Theorem 2]{Cho_2011_disperEst}: 
\begin{equation}
    \left\||D_x|^{ -d(1-\frac{\A}{2})(\frac{1}{2}-\frac{1}{r})} U_{\A} (t) \phi\right\|_{L^q(\R ;L^r(\R^d))} \le C\|\phi \|_{L^2(\R^d)},
    \label{L2Str}
\end{equation}
where $\A\in (0,\infty)\setminus \{ 1\}$ and $q,r$ satisfy $2\le q,r\le \infty$, $2/q+d/r=d/2$ and $(q,r,d)\neq (2,\infty,2)$ and $|D_x|^s f$ is the homogeneous derivative of order $s,\,s\in\R$, that is, $$\widehat{|D_x|^s f} (\xi) =c |\xi |^s \hat{f} (\xi)$$ and \begin{equation*}
    \left[U_{\A}(t)\phi\right] (x) =\int_{\R^d}
    e^{ix\xi -it|\xi|^{\A}} \hat{\phi} (\xi )d \xi.
\end{equation*}
Here $\hat{\phi}$ denotes the Fourier transform of $\phi.$
As far as the authors know, the Strichartz type estimate for the fractional Schr\"odinger equation can be traced back to  \cite[Theorem 2.1]{Kenig_1991_OsciIntDisp}, where the authors proved  \eqref{L2Str} for $\A\neq 1$ in one space dimension.  In particular, the diagonal case of \eqref{L2Str} in one space dimension is
\begin{equation}
    \||D_x|^{\frac{\A-2}{6} } U_{\A}(t) \phi \|_{L^{6}_{xt}(\R^2)} \le C\| \phi \|_{L^2(\R)}. \label{dSFS}
\end{equation}
This study is motivated by the problem of extending (\ref{dSFS}) for data $\phi\notin L^2$.  For the Cauchy problem of the standard nonlinear Schr\"odinger equation, there have been several attempts to establish well-posedness in non-$L^2$-based spaces.  In  \cite{Cazenave_2001_NLS,Hyakuna_2012_FLp_NLS,Grunrock_2005_BiTriEst}, local well-posedness of the 1D nonlinear Schr\"odinger equation was established in the \textbf{Fourier--Lebesgue space} $\Lp$ and \textbf{Fourier--Sobolev space} $\HSP$ defined by
\begin{equation*}
    \Lp (\R) :=\{\phi \in \mathcal{S}'(\R)\,| \hat{\phi } \in L^{p'}(\R) \}
\end{equation*}
and 
\begin{equation*}
    \HSP (\R):= \{ \phi \in \mathcal{S}'(\R) \,| (1+\xi^2)^{s/2} \hat{\phi}(\xi) \in L^{p'}(\R)\,\},
\end{equation*}
respectively, where $1/p+1/p'=1$(throughout the paper, primes denote conjugate exponents).  A key estimate to the local well-posedness results in these earlier works is
\begin{equation}
    \| U_2(t) \phi \|_{L^{3p}_{xt}(\R^2)} \le C\| \phi \|_{\Lp},\label{SFS}
\end{equation}
which holds true for $p>4/3$.  This can be viewed as a generalization of the standard (diagonal) Strichartz estimate and is also called the \textbf{Fefferman--Stein estimate}, see  \cite{Cazenave_2001_NLS, Fefferman_1970_SingConvo} or the introduction in \cite{Grunrock_2004_LWmKdV}.  The estimate plays an essential role also in the study of the nonlinear Schr\"odinger equations in non-$L^2$-based spaces such as 
Fourier-Lebesgue and Fourier Sobolev spaces.  Thus one may expect that a generalization of \eqref{SFS} to the fractional equation leads to the developent of  a well-posedness theory of the 1D fractional nonlinear Schr\"odinger equation in $\HSP$.  The first aim of this paper is to obtain a Fefferman--Stein type estimate for the fractional Schr\"odinger equation in order to establish well-posedness of the nonlinear Cauchy problem.  More precisely, we prove: 

\begin{thm} \label{FS}
Let $4/3<p < \infty$ and $\A \in (0,2] \setminus \{1 \}$.  Then the estimate
\begin{equation}
    \left\| |D_x|^{\frac{\A-2}{3p}} U_{\A}(t) \phi\right\|_{L^{3p}_{xt} (\R^2)} \le C\|\phi\|_{\Lp},\label{eqFS}
\end{equation}
holds true.
\end{thm}

Observe that putting $p=2$ in \eqref{eqFS} gives \eqref{dSFS}.  Thus Theorem \ref{FS} can be regarded as a generalization of the diagonal case of the Strichartz estimate for the one dimensional fractional Schr\"odinger equation.  Note also that putting $\A=2$ gives \eqref{SFS}, which implies \eqref{eqFS} is an extension of the Fefferman--Stein estimate to fractional equations.  To prove Theorem \ref{FS} we write $|U_{\A}\phi \overline{U_{\A}\phi}|^2$ as  the double integral
\begin{equation}
    U_{\A}(t) \phi \overline{ U_{\A}(t) \phi}=\iint_{\R\times \R} e^{ix(\xi_1-\xi_2)+it(|\xi_2|^{\A}-|\xi_1|^{\A})} 
\hat{\phi}(\xi_1) \overline{\hat{\phi}(\xi_2)} d\xi_1 d\xi_2.\label{introbilinear}
\end{equation}
Then, the $L^{3p/2}$-norm of this integral is estimated by the Hausdorff--Young inequality after changes of variables.  In particular, a suitable estimate of the associated Jacobian determinant leads to the derivatives appearing in \eqref{eqFS}, which is a key in the proof.  This kind of approach is well known.  For example, it is used to obtain the estimate of Fefferman and Stein type for the Airy equation in \cite{Masaki_2016_FLpKdV}.  It is easy to show \eqref{SFS} for the standard Schr\"odinger equation ($\A=2)$ by this method, since there is no gain or loss of derivatives there.  The case of the fractional Schr\"odinger equation  is more difficult than these cases.  Our key estimates to prove Theorem \ref{FS} are some estimates for the difference of power functions and a weighted version of the Hardy--Littlewood--Sobolev-type inequality established by Stein and Weiss \cite{SteinWeiss_1958_FrcInt}.  These estimates enable us to derive the desired loss of derivatives $|D_x|^{(\A-2)/3p}$ in \eqref{eqFS} from the Jacobian determinant and a bound for the norm of the bilinear form \eqref{introbilinear}  in terms of the $\Lp$-norm of data. 

We summarize additional comments on Theorem \ref{FS} below:

\begin{rem}
\begin{enumerate}
\item 
The estimate \eqref{eqFS} holds true for $p=\infty$ with the convention that $1/\infty=0$.  This case is trivial.
\item 
The condition  $p>4/3$ is required to apply the Hausdorff--Young inequality to obtain an $L^{3p/2}$-bound of the above double integral.  Indeed, the estimate fails if $p\le 4/3$ when $\A=2$ (see \cite{Grunrock_2005_BiTriEst}).  So we cannot expect \eqref{eqFS} to hold for $p \le 4/3$ for the fractional case, and thus we may call it  sharp in this sense.

\item In this paper, we focus only on the lower order cases $\A\le 2$ and we do not consider the case $\A>2$.  However, our approach to prove Theorem \ref{FS} can cover higher order cases.  Indeed, the estimate \eqref{eqFS} also holds true for $\A \in [2,3)$.  See Remark \ref{FSremark2} following the proof of Theorem \ref{FS}. The higher order fractional NLS is also interesting and it will be treated in our forthcoming work.
\item 
 It would be interesting to study estimates of the form \eqref{eqFS} of for more general dispersive equations.  When $p=2$, these dispersive estimates are established in \cite{Cho_2011_disperEst,Kenig_1991_OsciIntDisp}.  However, to our knowledge, there is no earlier works that generalized these estimate for $p\neq 2$.
 \item It is also interesting to consider estimates of type \eqref{eqFS} in higher space dimensions.  The case $p>2$ is straightforward.  Indeed, it follows immediately from interpolation between the trivial case $p=\infty$ (see (i)) and the standard Strichartz estimates for $\phi \in L^2$.  On the other hand, the case $p<2$ is much more difficult.  Indeed, not much is known about similar estimates for $p<2$ even in the case of the standard Schr\"odinger equation ($\A=2$).  See e.g. \cite{Cazenave_2001_NLS}. 
\end{enumerate}
\end{rem}

Off-diagonal generalizations of \eqref{eqFS} are easily obtained by interpolation with the standard Strichartz estimates \eqref{L2Str}: The estimate
\begin{equation}
    \| |D_x|^{\sigma} U_{\A}(t)\phi \|_{L^q_t (\R;L^r_x)} \le C\|\phi \|_{\Lp} \label{StrLqLr}
    \end{equation}
holds true for suitable $\sigma ,q,r$.  We do not give a precise statement of \eqref{StrLqLr} here.  In fact, we do not use this kind of estimate in our application to non-linear Cauchy problems because of the difficulty of a loss of derivatives.  In this paper, we present another kind of off-diagonal generalization of the Fefferman--Stein-type estimate.  Instead of the standard Strichartz space $L^q(I ; L^r)$, we consider $L^r_x L_t^q(I)$ spaces defined by
\begin{equation*}
    \|f \|_{L^r_x L_t^q(I)}
    := \left( \int_{\R} \| f(\cdot, x) \|_{L^q (I)}^r dx \right)^{\frac{1}{r}}
\end{equation*}
for $r<\infty$ and
\begin{equation*}
    \|f \|_{L^{\infty}_x L_t^q(I)}
    := {\rm ess\,sup}_{x\in \R}\| f(\cdot, x) \|_{L^q (I)},
\end{equation*}
where $I\subset \R$ is an interval.  A merit of using this kind of spaces is that we can exploit a gain of derivative in these functional framework.  Indeed, in \cite{Kenig_1991_OsciIntDisp}, Kenig, Ponce and Vega obtained the following smoothing effect:
\begin{equation}
\| |D_x|^{\frac{1}{2}}U_2(t) \phi\|_{L^{\infty}_x L_t^2 (\R)} \le C \|\phi \|_{L^2}.
\end{equation}
We have a similar smoothing property (see Lemma \ref{SmoothingL2}) for the fractional Schr\"odinger equation as long as $\A>1$:
\begin{equation}
\| |D_x|^{\frac{\A-1}{2}}U_{\A}(t) \phi\|_{L^{\infty}_x L_t^2 (\R)} \le C \|\phi \|_{L^2}. \label{introsmoothing}
\end{equation}
Interpolating this estimate with the Fefferman--Stein type inequality and the maximal function estimate (see Proposition \ref{maximalfunction} ), we get the
following off-diagonal generalizations of Fefferman--Stein estimate \eqref{FS}.
\begin{cor} \label{coroffGFS}
Let $q> 2,\,r> 4$ and $p>4/3$ satisfy
\begin{equation}
    \frac{1}{q}+\frac{2}{r}=\frac{1}{p}
    \label{scaling}
\end{equation}
and
\begin{equation*}
    0< \frac{1}{r} <\min \left(\frac{1}{4},\, \frac{1}{2}-\frac{1}{q}\right).
\end{equation*}
    Then, the estimate
    \begin{equation}
        \left\| |D_x|^{\sigma} U_{\A}(t) \phi\right\|_{L_x^r(L_t^q)} \le C\|\phi\|_{\Lp} \label{OffdGFS}
    \end{equation}
    holds true, where
    \begin{equation}
    \sigma=\frac{\A}{q}+\frac{1}{r}-\frac{1}{p}. 
    \label{dgain}
    \end{equation}
\end{cor}

Corollary \ref{coroffGFS} will play a vital role in our well-posedness results.  We conclude this subsection by making the following definition. 
\begin{defn}[$(\A,p)$-acceptability]
 Let $p>4/3$.  We say that a triplet $(q,r,\sigma)$ is $(\A,p)$-acceptable if $q,r,\sigma$ satisfy the assumption of Corollary \ref{coroffGFS}.  
\end{defn}

It would be clear that \eqref{dgain} is necessary from scaling considerations.  Indeed, it follows immediately by replacing $\phi$ with $\phi_{\lambda} (x):=\phi (\lambda x),\,\lambda >0$.  \eqref{scaling} arises from the inequalities used in the interpolation.  In the rest of the paper, for convenience, we occasionally refer to \eqref{scaling} and \eqref{dgain} as the ``scaling conditions" of $(\A,p)$-acceptability.  In Figure \ref{fig:acceptableintro} below, a triplet $(q,r,\sigma)$ satisfying \eqref{dgain} is $(\A,p)$-acceptable for some $p>4/3$ if $(1/r,1/q)$ lies in the shaded area i.e., inside the trapezoid.  Although \eqref{OffdGFS} holds true for some $(1/r,1/q)$ on the boundary of the trapezoid (for example $(1/2,0)$ corresponds to \eqref{introsmoothing}), we exclude them from the acceptability for simplicity.


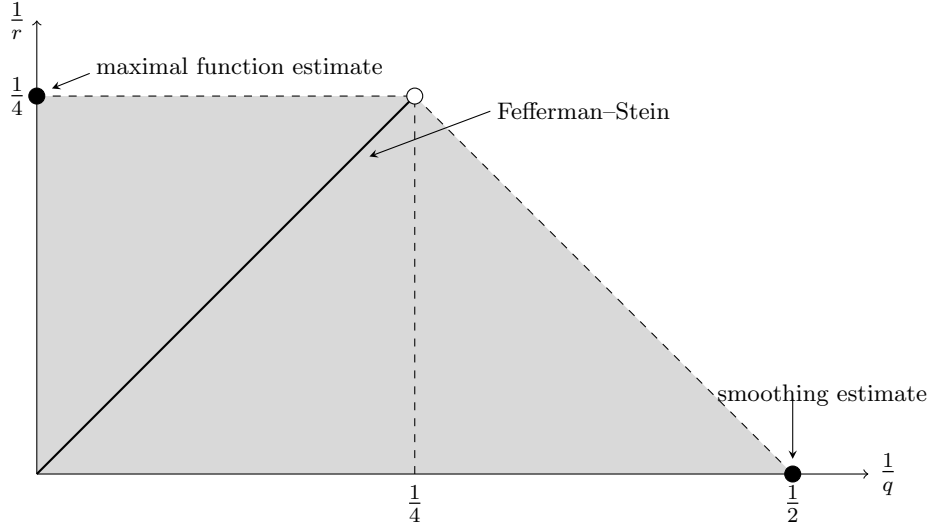
\begin{figure}[htbp]
  \centering

\begin{tikzpicture}[scale=20]

\fill[gray!30] (0,0) -- (0.5,0) -- (0.25,0.25) -- (0,0.25) -- cycle;

\draw[->] (0,0) -- (0.55,0) node[right] {$\frac{1}{q}$};
\draw[->] (0,0) -- (0,0.30) node[left] {$\frac{1}{r}$};

\draw (0.5,0) node[below] {$\frac12$};
\draw (0.25,0) node[below] {$\frac14$};
\draw (0,0.25) node[left] {$\frac14$};

\draw[dashed] (0,0.25) -- (0.25,0.25);
\draw[dashed] (0.25,0.25) -- (0.5,0);
\draw[dashed] (0.25,0) -- (0.25,0.25);

\draw[thick]
  (0,0) -- (1/4,1/4); 




\filldraw[fill=white,draw=black] ({1/4},{1/4}) circle[radius=0.15pt];


\filldraw[fill=black,draw=black] ({0},{1/4}) circle[radius=0.15pt];

\filldraw[fill=black,draw=black] ({1/2},{0}) circle[radius=0.15pt];

----------------------------------

\draw[->,>=stealth] (0.3,0.24) -- (0.22,0.21);

\draw (0.37,0.24) node[left,xshift=32pt] {\scriptsize Fefferman--Stein};


----------------------------------
\draw[->,>=stealth] (0.035,0.265) -- (0.01,0.255);

\draw (0.18,0.27) node[left,xshift=32pt] {\scriptsize maximal function estimate};

----------------------------------
\draw[->,>=stealth] (0.5,0.05) -- (0.5,0.01);

\draw (0.54,0.052) node[left,xshift=32pt] {\scriptsize smoothing estimate};

\end{tikzpicture}

\caption{$(q,r,\sigma)$ is $(\A,p)$-acceptable for some $\sigma$ and $p$ if $(1/q,1/r)$ lies in the shaded area.}
  \label{fig:acceptableintro}
\end{figure}


\subsection{1D Fractional NLS }
Once the generalized Strichartz type estimate \eqref{eqFS} is established, our second aim is to obtain well-posedness results in $\HSP$ for the Cauchy problem of the one dimensional fractional nonlinear Schr\"odinger equation with the
pure power type nonlinearity
\begin{equation}
iu_t+(-\pa_{xx})^{\A/2}u +|u|^{\rho-1}u=0,\quad u|_{t=0}=\phi. \label{FNLS}
\end{equation}

Although \eqref{eqFS} is a key to the well-posedness of \eqref{FNLS} in $\HSP$, there are still difficulties.  One of them is the existence of a loss of derivatives in the estimate.  Because of this, it is not easy to apply the usual Strichartz technique, which is well known for solving the standard nonlinear Schr\"odinger equations.  As is easily seen from \eqref{L2Str}, this difficulty already appears in the case $p=2$.  In \cite{HongSire_2015_frNLS} Hong and Sire were able to establish well-posedness results in $H^s(=\widehat{H^s_2})$ for the power type fractional NLS.  Their approach relies on a direct $H^s$ estimate of the nonlinearity and Sobolev's embedding.  Thus, one possible way to obtain well-posedness in $\HSP$ is to generalize their approach to the Fourier Lebesgue and Sobolev setting.  However, there is another difficulty.  It is not easy to extend their key nonlinear estimates and Sobolev inequality to the non-$L^2$-based spaces.  In fact, the authors' do not know whether or not similar nonlinear estimates hold true in the Fourier Sobolev spaces with $p\neq 2$, and it is known that some Sobolev type inequalities fail in the Fourier Lebesgue and Sobolev spaces (see e.g. \cite{Masaki_2016_FLpKdV}).  In this paper, we give up extending their approach and employ another functional framework.  Instead of the standard Strichartz space $L^q(I ; L^r)$ or $L^q(I; H^s_r)$, we establish a solution of \eqref{FNLS} in the aforementioned $L^r_x L_t^q(I)$ spaces.  A merit of using this kind of spaces is that we can take advantage of a gain of derivatives.  Indeed, thanks to the smoothing estimate \eqref{introsmoothing}, taking $\sigma>0$ is allowed for some choices of $q,r,p$.  Well-posedness in non-$L^2$-based Fourier--Sobolev spaces $\HSP$ is obtained from \eqref{eqFS} via the duality argument.  Let us briefly explain the strategy.  By duality, the Fefferman--Stein estimate   is equivalent to
\begin{equation}
\sup_I \left\| \int_I U_{\A}(-s) F(s)ds\right\|_{\widehat{L^{p'}} } \le C\||D_x|^{-\sigma}F\|_{L_x^{r'}L_t^{q'}}, \label{GFSdual}
\end{equation}
where the supremum is taken over all intervals on $\R$ (see Proposition \ref{inhomoest}). This gives us an estimate of the $\HSP$-norm of the nonlinear part of the integral equation equivalent to \eqref{FNLS} when $p<4$.  Consequently, the Fefferman--Stein type estimate \eqref{eqFS} and its dual along with inhomongeneous Strichartz type estimates in $L_x^rL_t^q$ spaces (see Proposition \ref{inhomoest}) and the trivial identity $\|U_{\A}(t)\phi \|_{\HSP}=\|\phi \|_{\HSP}$ enable us to estimate the integral equation in the space $L^{\infty}([0,T] ;\HSP)\cap L_x^rL_t^q([0,T])$ as long as $4/3<p<4$.  This kind of argument was used by M. Tsutsumi and the second author \cite{Hyakuna_2012_FLp_NLS} to obtain the well-posedness of the standard pure power NLS in $\Lp$ spaces in the framework of the standard Strichartz spaces $L_t^q([0,T]; L^r_x)$.  The one in $L_x^rL_t^q$ spaces was developed by Masaki and Segata in \cite{Masaki_2016_FLpKdV} to prove small data well-posedness for generalized Kortweg--de Vries equation in critcal $\Lp$ spaces and we follow their approach in this paper.  Note also that these spaces were also used in earlier studies \cite{Kenig_1993_SmallSol,kenig_1993_contraction} to solve derivative nonlinear Schr\"odinger equations in $H^s(=\widehat{H^s_2})$.  Perhaps, in these works $L_x^r L_t^q$ spaces were introduced to handle derivatives in nonlinearities.  Thus, we believe our idea is novel in the sense that we introduce these function spaces to overcome the difficulty of the loss of derivatives in the Strichartz estimate in the standard $L^q_t L^r_x$-spaces. 
Finally, our approach relies on the smoothing property of the Strichartz estimates of $L_x^rL_t^q$-type and it cannot work for $\A<1$.  Indeed, it is clear that one has a gain of derivatives in \eqref{introsmoothing} only if $\A>1$.

\subsection{Statement of the  Well-posedness Results}
We need some notations and numbers in order to state the well-posedness results.  The  scaling critical exponent $s_c$ is given by
\begin{equation*}
s_c=s_c(\A,p,\rho):=\frac{1}{p}-\frac{\A}{\rho-1}.
\end{equation*}
$s_c$ is determined in the usual way.  Indeed, if we let $u(t,x)$ be a solution of \eqref{FNLS}, then $u_{\lambda}(t,x):=\lambda^{\A/(\rho-1)} u(\lambda^{\A}t, \lambda x),\,\lambda>0$  also solves the same equation with the datum $u_{\lambda}(x):=\lambda^{\A/(\rho-1)} u(0,\lambda x)$.  The exponent $s_c$ is defined by
$\| u_{\lambda}(0)\|_{\dot{H}^{s_c}_p} =\|u(0)\|_{\dot{H}^{s_c}_p}$.  We note that $\HSP$ scales like $H^s_p$.
We introduce two more exponents:
\begin{equation*}
    s_l:=\frac{\A}{\rho-1}-\frac{\A-1}{p}
\end{equation*}
and
\begin{equation*}
    s_u:=\min \left(\frac{1}{2p},\,\,\frac{2\A p-p-4\A+4}{4p},\,\, \frac{6\A p-3p-4\A+2}{2p(\rho-1)}-\frac{\A-1}{p},
    \,\, \frac{2\A(2\A-1)}{p(\rho-1)}-\frac{\A-1}{p} \right).
\end{equation*}

Note that under the assumptions of the main theorems below, we have 
\begin{equation*}
s_u = \min \left(\frac{2\A p-p-4\A+4}{4p},\,\,\frac{6\A p-3p-4\A+2}{2p(\rho-1)}-\frac{\A-1}{p} \right) \quad \text{when}  \ \ p \le 2.
\end{equation*}
and
\begin{equation*}
  s_u = \min \left(\frac{1}{2p},\,\,\frac{2\A(2\A-1)}{p(\rho-1)}-\frac{\A-1}{p} \right)  \quad \text{when}  \ \ p \ge 2
\end{equation*}
In particular, $s_u=1/4$ when $p=2$.

In this paper, we establish well-posedness results for \eqref{FNLS} in Fourier--Sobolev and similar spaces for $s\in [s_l,s_u)$ under the assumption that $1<\A\le 2$, and
  \begin{equation}
        \max \left(\frac{6\A-1}{2\A-1},\,\, \frac{2\A-1+2\A p}{2\A-1} \right) < \rho \le 2p+1,\quad 
    \frac{2\A}{2\A-1} <p <\min ( 4\A-2,\,4).
    \label{rhoprange}
\end{equation}
Note that $s_l<s_u$ if $1<\A\le 2$ and \eqref{rhoprange} are satisfied (see Remark \ref{localremark} below).

Special exponents and numbers $s_0,q_0,q_1,r_0,r_1$ appear in the statement of the theorems and their corollaries below.  \,$s_0$ is a nonnegative constant satisfying suitable conditions.  $q_0,q_1,r_0,r_1$ are exponents depending on $\A,p,\rho,s$ and $s_0$ such that the two triplets $(q_0,r_0,-s)$ and $(q_1,r_1,s_0-s)$ are $(\A,p)$-acceptable.  The precise definition of these numbers is somewhat involved and will be given in Section \ref{preliminaries}.
Now we are ready to state the well-posedness results.  We begin with
local well-posedness for large data in \textbf{homogeneous Fourier Sobolev spaces} $\HHSP$ for $s\ge s_l$ defined by
\begin{equation*}
\HHSP(\R) := \{ \phi \in \mathcal{S}'(\R)\,\,|\, |\xi|^s \hat{\phi}(\xi) \in L^{p'}(\R)\,\}.
\end{equation*}

In what follows, we write $I_T:=[0,T]$ for $T>0$.
\begin{thm} \label{LWPsubcriticalp}
    Let $1<\A\le2$.  Assume that $\rho,p,\A$ satisfy \eqref{rhoprange}.  Assume moreover that $(\rho ,s )\neq (2p+1,s_l)$ and 
\begin{equation}
    s_l \le s < s_u \label{subcriticalcondp}
\end{equation}
Then, for any $\phi \in \HHSP$ there exist $T:=T(\|\phi \|_{\HHSP})>0$ and a unique solution $u \in C([0,T] ; \HHSP )$ of \eqref{FNLS} such that
\begin{equation}
    u \in L^{r_0}_x L_t^{q_0}(I_T),\quad \text{and}\quad |D_x|^{s_0} u \in L_x^{r_1}L_t^{q_1}(I_T).
\end{equation}
Moreover, 
\begin{equation}
    |D_x|^{s+\sigma} u \in L^{r}_x L^{q}_t (I_T) \label{Strregularity1}
\end{equation}
for any $(\A,p)$-acceptable triplet $(q,r,\sigma)$. 

Furthermore, continuous dependence on the initial data holds in the following sense.  For any $M>0$ there exists $T_M>0$ with the following properties: if $u_1,u_2$ denote two solutions of \eqref{FNLS} with $u_j(0)=\phi_j,\,\|\phi_j\|_{\HHSP} \le M,\,j=1,2$, then one has
\begin{equation}
\sup_{t\in I_M} \| u_1(t)-u_2(t)\|_{\HHSP} \le C \|\phi_1 -\phi_2\|_{\HHSP}.
\end{equation}

\end{thm}

\medskip

Noting that 
$\HSP \subset \widehat{\dot{H}^{s_l}_p} \cap\HHSP  $, we can derive local well-posedness in inhomogeneous spaces $\HSP$ from Theorem \ref{LWPsubcriticalp}:

\begin{cor} \label{LWPinhomogeneous}
Let $1<\A \le 2$ and let $p,\rho,s$ be as in Theorem \ref{LWPsubcriticalp}.  Then, for any $\phi \in \HSP$ there exist $T:=T(\|\phi \|_{{\widehat{\dot{H}^{s_l}_p}}}, \|\phi \|_{\HHSP})>0$  and a unique solution $u \in C([0,T] ; \HSP )$ of \eqref{FNLS} such that
\begin{equation}
    u \in L^{r_0}_x L_t^{q_0}(I_T),\quad \text{and}\quad |D_x|^{s_0} u \in L_x^{r_1}L_t^{q_1}(I_T).
\end{equation}

Moreover, 
\begin{equation}
    |D_x|^{s+\sigma} u \in L^{r}_x L^{q}_t (I_T) \label{Strregularitya}
\end{equation}
for any $(\A,p)$-acceptable triplet $(q,r,\sigma)$.

Furthermore, continuous dependence on the initial data holds in the following sense.  For any $M>0$ there exists $T_M>0$ with the following properties: if $u_1,u_2$ denote two solutions of \eqref{FNLS} with $u_j(0)=\phi_j,\,\|\phi_j\|_{\HHSP} \le M,\,j=1,2$, then one has
\begin{equation}
\sup_{t\in I_M} \| u_1(t)-u_2(t)\|_{\HSP} \le C \|\phi_1 -\phi_2\|_{\HSP}.
\end{equation}

\end{cor}

\begin{rem} \label{localremark}
\begin{enumerate}
\item 
The existence of the upper bound $s_u$ for the regularity comes from the availability of exponents that allow both the linear and nonlinear estimates for the corresponding integral equation, which will be needed to construction a solution of \eqref{FNLS}.  See Subsection \ref{Spexponent}.  This condition seems to be of a technical nature and could be removed or relaxed(we do not discuss this here, though). 
    \item
    We have $s_l <s_u$ under the assumption $1<\A\le 2$ and \eqref{rhoprange}.  Thus the range \eqref{subcriticalcondp} in Theorem \ref{LWPsubcriticalp} is not empty.  Indeed, it is easy to check that
\begin{align*}
&s_l< \frac{1}{2p} \quad \text{if and only if} \quad \rho >\frac{2\A-1+2\A p}{2\A-1}, \\
&s_l < \frac{2\A p-p-4\A+4}{4p} \quad \text{if and only if} \quad \rho>\frac{6\A-1}{2\A-1},  \\
&s_l <\frac{6\A p-3p-4\A+2}{2p(\rho-1)}-\frac{\A-1}{p} \quad \text{if} \quad p>\frac{2\A}{2\A-1}, \\
&s_l <\frac{2\A(2\A-1)}{p(\rho-1)}-\frac{\A-1}{p} \quad  \text{if and only if} \quad p<4\A-2.
\end{align*}

\item Our results state that \eqref{FNLS} is locally well posed in $\HSP,\HHSCP$ for $s \ge s_l$.  The condition may seem unusual.  This also comes from technical reasons; it is necessary to estimate some derivative of the nonlinear term in $L_x^rL_t^q$ spaces. 
\item The exponent $s_c$ is well known.  Scaling considerations suggest well-posedness of \eqref{FNLS} in $\HSP$ could be pushed down to $s\ge s_c$.  On the other hand, in \cite{HongSire_2015_frNLS}, Hong and Sire proposed another threshold $s_g:=\frac{2-\A}{4}$ for the well-posedness of \eqref{FNLS} in the case of $p=2$. They conjectured that \eqref{FNLS} is locally well posed in $H^s$ for $s\ge \max (s_c,s_g)$ and it is ill-posed for $s<\max (s_c,s_g)$.  Indeed, Cho et al.\cite{Cho_2015_well_ill_NLS} showed ill-posedness of the 1D cubic fractional NLS in the sense that the map $\phi \mapsto u$ is not uniformly continuous from $H^s$ to $C([0,T]\, ;H^s)$.  For $p=2$, it is easy to see that $s_l \ge \max (s_c, s_g)$, with the equality if and only if $\rho =5$.  In this sense, our results leave room for improvement.  This will be studied in our future work (it seems functional framework other than $L_x^rL^q$ type spaces should be employed in order to remove the condition $s\ge s_l$). For the case of general $p$, it is also easy to see that $s_l \ge s_c$ and equality holds if and only if $\rho =2p+1$.  Thus one can reach a critical $\HHSP$-space when $\rho =2p+1$.  This leads to the small data global well-posedness results at the critical Sobolev regularity below.

\item 
Putting $\A=p=2$ in Theorem \ref{LWPsubcriticalp} and Corollary \ref{LWPinhomogeneous} yields well-posedness of the standard 1D NLS in $H^s$.  However, the results do not  cover all the case of the known well-posedness result (see e.g. \cite{Cazenave_2003_SemilinearSchrodinger, Cazenave_1990_criticalsobolev}).  For example, it is clear the our result does not cover the cubic nonlinearity.  While we can take advantage of the gain of derivatives, the condition on available exponents in the framework of $L_x^r L_t^q$ appears to be more restrictive than that in the standard Strichartz space $L_t^qL^r_x$.

\end{enumerate}
\end{rem}

As mentioned above, $s_l$ reaches the scaling critical exponent $s_c$ when $\rho =2p+1$, for which we can establish small data global well-posenenss results in $\HSCP,\HHSCP$:

\begin{thm}\label{SmalldataGWP}
    Let $1<\A\le2,\quad \rho=2p+1$, and 
    \begin{equation*}
    s_c=s_c(\A,p,2p+1)=\frac{2-\A}{2p}.
    \end{equation*}
    Assume that 
    \begin{equation*}
        \frac{2\A}{2\A-1} <p < \min (4, 4\A-2).
    \end{equation*}
There is an $\varepsilon>0$ such that for any $\phi \in \HHSCP$ with $\|\phi \|_{\HHSCP}< \varepsilon$ there exists a unique global solution $u \in C(\R ; \HHSCP) \cap L^{\infty}( \R ; \HHSCP) $ of (\ref{FNLS}) satisfying
\begin{equation*}
    u \in L^{r_0}_x L_t^{q_0}(\R),\quad \text{and}\quad |D_x|^{s_0} u \in L_x^{r_1}L_t^{q_1}(\R).
\end{equation*}
Moreover, 
\begin{equation}
    |D_x|^{s+\sigma} u \in L^{r}_x L^{q}_t (\R) \label{Strregularityb}
\end{equation}
for any $(\A,p)$-acceptable triplet $(q,r,\sigma)$.   
\end{thm}

\begin{cor}\label{SmalldataGWPinhomo}
    Let $1<\A\le 2,\quad \rho=2p+1$, and 
    \begin{equation}
    s_c=s_c(\A,p,2p+1)=\frac{2-\A}{2p}.
    \end{equation}
    Assume that 
    \begin{equation}
        \frac{2\A}{2\A-1} <p < \min (4, 4\A-2).
    \end{equation}

There is an $\varepsilon>0$ such that for any $\phi \in \HSCP$ with $\|\phi \|_{\HHSCP}< \varepsilon$ there exists a unique global solution $u \in C(\R ; \HSCP) \cap L_{loc}^{\infty}( \R ; \HSCP) $ of (\ref{FNLS}) satisfying
\begin{equation}
    u \in L^{r_0}_x L_t^{q_0}(\R),\quad \text{and}\quad |D_x|^{s_0} u \in L_x^{r_1}L_t^{q_1}(\R).
\end{equation}

Moreover, 
\begin{equation}
    |D_x|^{s+\sigma} u \in L^{r}_x L^{q}_t (I_T) \label{Strregularityc}
\end{equation}
for any $(\A,p)$-acceptable triplet and $T>0$.    
\end{cor}
\begin{rem}
 The persistence property of the small global solutions in the inhomogeneous spaces in the above corollary is weaker than those in homogeneous spaces in the sense that we can only show that $u\in L_{loc}^{\infty}(\R ; \HSCP)$, not $u\in L^{\infty}(\R ; \HSCP)$.  This is because the $\Lp$-regularity of solution that leads $\HSCP$-regularity is shown only for finite time intervals.  However, we have $u\in L^{\infty} (\R ; \HSP)$ for $\rho=5,\,p=2$ which corresponds to the following quintic fractional NLS in the critical space
 \begin{equation}
     iu_t+(-\pa_{xx})^{\A/2}u+|u|^4u=0,\quad u|_{t=0}=\phi \in H^{\frac{2-\A}{4}}. \label{quintic}
 \end{equation}
 For details, see the last of the proof of the corollary in Section \ref{proofinhomo}.  Note also that, it seems small data global well-posedness of \eqref{quintic} in the critical space $H^{\frac{2-\A}{4}}$ is not covered by the standard well-posedness results in $H^s$ in \cite{HongSire_2015_frNLS}.  So we believe that our results is new even with respect to the most typical case of $p=2$.
\end{rem}

\par
This paper is organized as follows.  In Section \ref{FSproof}, we prove the Fefferman--Stein-type estimate (Theorem \ref{FS}) and its off-diagonal generalization.  Section \ref{proofLWP} is devoted to the proof of the well-posedness results in homogeneous Fourier--Sobolev spaces.  Finally, in Section \ref{proofinhomo}, we prove the well-posedness results in the inhomogeneous spaces. The results are deduced as corollaries of that in homogeneous spaces; it can be proved by showing the $\Lp$-regularity of the $\widehat{\dot{H}^{s_l}_p}$-solution and some uniqueness results.  The uniqueness assertion is established by Sobolev type embedding $L^r_xL_t^q$-spaces and regularity \eqref{Strregularity1}, \eqref{Strregularityc} of the solution and its derivatives in for arbitrary $(\A,p)$-acceptable  $L_x^rL_t^q$-spaces.

\section{Proof of Theorem \ref{FS} and Corollary \ref{coroffGFS}} \label{FSproof}
In this section, we mainly prove the Fefferman--Stein-type estimate.  Corollary \ref{coroffGFS} is an immediate consequence of the smoothing estimate and the maximal function estimate via interpolation.  Although these estimates are known, we give their proof in Subsection \ref{FScorproof} for completeness.  
\subsection{Proof of the Fefferman--Stein-type estimate } \label{subsectionFS}
We first present two key inequalities.  One is elementary estimates for the difference of power functions:

\begin{lem} \label{JestL}
Let $\beta \neq 0$.  Then
\begin{equation}
\left|\xi_1^{\beta}-\xi_2^{\beta}\right|^{-1} \le C^1_{\beta}\, \frac{\,\xi_1^{1-\beta} +\xi_2^{1-\beta}}{|\xi_1-\xi_2|} \label{Jest1}
\end{equation}
and
\begin{equation}
\left|\xi_1^{\beta}-\xi_2^{\beta}\right|^{-1} \le C^2_{\beta}\, \frac{\,\xi_1 \xi_2^{-\beta} +\xi_2 \xi_1^{-\beta}}{|\xi_1-\xi_2|} \label{Jest2}
\end{equation}
for any $\xi_1,\xi_2>0$ with $\xi_1\neq \xi_2$.

\end{lem}

\begin{proof}
(\ref{Jest1}) is straightforward.  Indeed, noting that
\begin{equation*}
|\xi_1^{\beta}-\xi_2^{\beta}| =\left|\beta \int^{\xi_1}_{\xi_2} \eta^{\beta-1} d\eta \right| \ge |\beta| \min (\xi_1^{\beta-1}, \xi_2^{\beta-1} ) \,|\xi_1 -\xi_2|,
\end{equation*}
we have
\begin{equation*}
\left|\xi_1^{\beta}-\xi_2^{\beta}\right|^{-1} \le |\beta|^{-1} \, \max ( \xi_1^{1-\beta},\, \xi_2^{1-\beta} ) \, |\xi_1-\xi_2|^{-1}
\le |\beta|^{-1} \, \frac{ \xi_1^{1-\beta} +\xi_2^{1-\beta}} {|\xi_1-\xi_2|}.
\end{equation*}
(\ref{Jest2}) follows from (\ref{Jest1}).  Indeed, we have
\begin{eqnarray*}
\left|\xi_1^{\beta}-\xi_2^{\beta}\right|^{-1} &=& \left| \frac{\xi_1^{-\beta} -\xi_2^{-\beta}}{ \xi_1^{-\beta}\xi_2^{-\beta}} \right|^{-1} \\
&\le &\xi_1^{-\beta} \xi_2^{-\beta} \times C_{-\beta}^1 \, \frac{ \xi_1^{1+\beta} +\xi_2^{1+\beta}} {|\xi_1-\xi_2|} \\
&=& C_{-\beta}^1 \frac{\,\xi_1 \xi_2^{-\beta} +\xi_2 \xi_1^{-\beta}}{|\xi_1-\xi_2|}.
\end{eqnarray*}

\end{proof}

 The other key estimate is the weighted Hardy--Littlewood--Sobolev inequality due to
Stein and Weiss \cite{SteinWeiss_1958_FrcInt}:

\begin{prop}( \cite[Theorem $B_2^{\ast}$]{SteinWeiss_1958_FrcInt})\label{WHLS} \quad Let $0<\lambda<n,\,1<q_1\le q_2<\infty,\,k_1<n/q_1',\,k_2<n/q_2,\,k_1+k_2\ge 0$ and
\begin{equation*}
\frac{1}{q_2}=\frac{1}{q_1}+\frac{\lambda+k_1+k_2}{n}-1.
\end{equation*}
Then
\begin{equation*}
\left| \int_{\R^n}\int_{\R^n} \frac{f_1(y)f_2(x)}{|x|^{k_2} |x-y|^{\lambda} |y|^{k_1}} dy dx\right|
\le C\|f_1\|_{L^{q_1}(\R^n)} \|f_2\|_{L^{q_2'}(\R^n)}.
\end{equation*}
\end{prop}

\begin{proof}[\textbf{Proof of Theorem \ref{FS}}]
 Let $r>4$.  We write
\begin{eqnarray*}
\| \U \phi \|_{L_{xt}^r(\R^2)}^2 &=& \left\| |\U \phi|^2 \right\|_{L^{\frac{r}{2}}_{xt} (\R^2)}.
\end{eqnarray*}
We have
\begin{equation*}
|\U f(x)|^2=\iint_{\R\times \R} e^{ix(\xi_1-\xi_2)+it(|\xi_2|^{\A}-|\xi_1|^{\A})} 
\hat{\phi}(\xi_1) \overline{\hat{\phi}(\xi_2)} d\xi_1 d\xi_2:=K_1+K_2+K_3+K_4,
\end{equation*}
where
\begin{eqnarray*}
K_1&:= &\iint_{\xi_1,\xi_2>0} e^{ix(\xi_1-\xi_2)+it(|\xi_2|^{\A}-|\xi_1|^{\A})} 
\hat{\phi}(\xi_1) \overline{\hat{\phi}(\xi_2) }d\xi_1 d\xi_2 \\
K_2&:= &\iint_{\xi_1>0,\xi_2<0} e^{ix(\xi_1-\xi_2)+it(|\xi_2|^{\A}-|\xi_1|^{\A})} 
\hat{\phi}(\xi_1) \overline{\hat{\phi}(\xi_2) }d\xi_1 d\xi_2 \\
K_3&:= &\iint_{\xi_1<0,\xi_2>0} e^{ix(\xi_1-\xi_2)+it(|\xi_2|^{\A}-|\xi_1|^{\A})} 
\hat{\phi}(\xi_1) \overline{ \hat{\phi}(\xi_2)} d\xi_1 d\xi_2 \\
K_4&:= &\iint_{\xi_1,\xi_2<0} e^{ix(\xi_1-\xi_2)+it(|\xi_2|^{\A}-|\xi_1|^{\A})} 
\hat{\phi}(\xi_1) \overline{\hat{\phi}(\xi_2) }d\xi_1 d\xi_2. \\
\end{eqnarray*}
We only estimate $K_1$.  The other contributions can be treated in a similar manner.  We introduce a change of variables $\Psi:(\xi_1,\xi_2) \mapsto (\eta_1,\eta_2)$ by
\begin{equation*}
\eta_1=\xi_1-\xi_2,\quad \eta_2=\xi_2^{\A}-\xi_1^{\A}.
\end{equation*}
Then its Jacobian determinant is given by:
\begin{equation*}
\frac{\pa(\eta_1,\eta_2)}{\pa(\xi_1,\xi_2)}
=\begin{vmatrix}
1 &-1 \\
-\A\xi_1^{\A-1} & \A\xi_2^{\A-1} 
\end{vmatrix}
=\A(\xi_2^{\A-1}-\xi_1^{\A-1}).
\end{equation*}
We put
\begin{equation*}
J(\xi_1,\xi_2):=\left| \frac{\pa(\eta_1,\eta_2)}{\pa(\xi_1,\xi_2)}\right|= \A\left|\xi_1^{\A-1}-\xi_2^{\A-1} \right|,
\end{equation*}
and 
\begin{equation*}
G(\eta_1,\eta_2) := J(\xi_1(\eta_1,\eta_2), \xi_2(\eta_1,\eta_2))^{-1} \hat{\phi}(\xi_1(\eta_1,\eta_2)) \overline{\hat{\phi}(\xi_2(\eta_1,\eta_2))}\mathbf{1}_R(\eta_1,\eta_2),
\end{equation*}
where $R:=\Psi(\{ \xi_1,\xi_2>0 \} )$ and $\mathbf{1}_R$ denotes indicator function.

Then 
\begin{equation*}
K_1=\iint_{\R\times \R} e^{ix\eta_1+it\eta_2} G(\eta_1,\eta_2) d\eta_1 d\eta_2
=\mathcal{F}^{-1} G(x,t).
\end{equation*}

Now, by the Hausdorff--Young inequality we have
\begin{equation*}
\|K_1\|_{L^{\frac{r}{2}}_{xt}} \le C\|G\|_{L^{\frac{r}{r-2}}_{\eta_1\eta_2}}.
\end{equation*}

We estimate the norm on the right hand side of the inequality.  By the change of variables $(\eta_1,\eta_2) \mapsto (\xi_1,\xi_2)$ with the Jacobian determinant $J(\xi_1,\xi_2)$ we have 
\begin{eqnarray*}
\|G\|_{L^{\frac{r}{r-2}}_{\eta_1\eta_2}}^{\frac{r}{r-2}}
&=&\iint_{\R\times \R} |G(\eta_1,\eta_2)|^{\frac{r}{r-2}} d\eta_1 d\eta_2 \\
&=& \iint_{R} |J(\xi_1(\eta_1,\eta_2),\xi_2(\eta_1,\eta_2))|^{-\frac{r}{r-2}} 
|\hat{\phi}(\xi_1(\eta_1,\eta_2)) \overline{\hat{\phi}(\xi_2(\eta_1,\eta_2))}|^{\frac{r}{r-2}} 
d\eta_1 d\eta_2\\
&=& \iint_{\xi_1,\xi_2>0} |J(\xi_1,\xi_2)|^{-\frac{2}{r-2}} 
|\hat{\phi}(\xi_1) \overline{\hat{\phi}(\xi_2)}|^{\frac{r}{r-2}}
d\xi_1 d\xi_2\\
&=& C \iint_{\xi_1,\xi_2>0} \left| \xi_1^{\A-1}-\xi_2^{\A-1} \right|^{-\frac{2}{r-2}} 
| \hat{\phi}(\xi_1) \overline{\hat{\phi}(\xi_2)} |^{\frac{r}{r-2}}
d\xi_1 d\xi_2.\\
\end{eqnarray*}

Now we use Lemma \ref{JestL} to estimate the difference derived from the Jacobian determinant.  We consider two cases.  
We first assume $1<\A\le 2$.  We apply (\ref{Jest1}) to obtain
\begin{equation*}
\left| \xi_1^{\A-1}-\xi_2^{\A-1} \right|^{-1} \le C\frac{|\xi_1|^{2-\A} +|\xi_2|^{2-\A}}{|\xi_1-\xi_2|}.
\end{equation*}
Then, noting the elementary inequailty $(a+b)^{\gamma} \le 2^{\gamma}(a^{\gamma}+b^{\gamma}),\,a,b>0,\gamma>0$, we have
\begin{equation*}
\left| \xi_1^{\A-1}-\xi_2^{\A-1} \right|^{-\frac{2}{r-2}}  
\le C\frac{|\xi_1|^{\frac{2(2-\A)}{r-2}}+|\xi_2|^{\frac{2(2-\A)}{r-2}} }{|\xi_1-\xi_2|^{\frac{2}{r-2}}}.
\end{equation*}
Therefore $\|G\|_{L^{\frac{r}{r-2}}_{\eta_1\eta_2}}^{\frac{r}{r-2}}$ is smaller than
\begin{align*}
&
\iint_{\R^2} \frac{|\xi_1|^{\frac{2(2-\A)}{r-2}} }{|\xi_1 -\xi_2 |^{\frac{2}{r-2}}}
\left| \hat{\phi}(\xi_1)
\overline{\hat{\phi}(\xi_2)} \right|^{\frac{r}{r-2}}d\xi_1d\xi_2+\iint_{\R^2} \frac{|\xi_2|^{\frac{2(2-\A)}{r-2}} }{|\xi_1 -\xi_2 |^{\frac{2}{r-2}}}
\left| \hat{\phi}(\xi_1)
\overline{\hat{\phi}(\xi_2)} \right|^{\frac{r}{r-2}} d\xi_1d\xi_2 \\
&:=I_1+I_2.
\end{align*}
We only consider $I_1$.  $I_2$ can be treated in a similar manner.  We write
\begin{eqnarray*}
I_1&=&\iint_{\R^2} \frac{ |\xi_1|^{\frac{2-\A}{r-2}} |\xi_2|^{-\frac{2-\A}{r-2}    }  } {|\xi_1-\xi_2|^{\frac{2}{r-2}}   }  \left(|\xi_1|^{\frac{2-\A}{r}} |\hat{\phi}(\xi_1)| \right)^{\frac{r}{r-2}} \left(|\xi_2|^{\frac{2-\A}{r}} |\hat{\phi}(\xi_2)| \right)^{\frac{r}{r-2}}  d\xi_1 d\xi_2 \\
&=&
\iint_{\R^2} \frac{  \left(|\xi_1|^{\frac{2-\A}{r}} |\hat{\phi}(\xi_1)| \right)^{\frac{r}{r-2}} \left(|\xi_2|^{\frac{2-\A}{r}} |\hat{\phi}(\xi_2)| \right)^{\frac{r}{r-2}}   } {   |\xi_1|^{-\frac{2-\A}{r-2}}  |\xi_1-\xi_2|^{\frac{2}{r-2}} |\xi_2|^{\frac{2-\A}{r-2}    }  }  d\xi_1 d\xi_2 .
\end{eqnarray*}

Then we apply Proposition \ref{WHLS} with
\begin{eqnarray*}
f_1&=&\left(|\xi_1|^{\frac{2-\A}{r}} |\hat{\phi}(\xi_1)| \right)^{\frac{r}{r-2}} ,\quad f_2=\left(|\xi_2|^{\frac{2-\A}{r}} |\hat{\phi}(\xi_2)| \right)^{\frac{r}{r-2}},\\
k_1&=&-\frac{2-\A}{r-2},\quad k_2=\frac{2-\A}{r-2},\quad \lambda=\frac{2}{r-2},\\
q_1&=&\frac{r-2}{r-3},\quad q_2=r-2,
\end{eqnarray*}
where it can easily be checked that these satisfy the assumption of Proposition if $r>4$. Especially, 
note that the condition for the power of the weights are satisfied since $1<\A$.

Consequently, we get
\begin{eqnarray*}
I_1 &\le &C \left\|  \left(|\xi_1|^{\frac{2-\A}{r}} |\hat{\phi}(\xi_1)| \right)^{\frac{r}{r-2}} \right\|_{L^{ \frac{r-2}{r-3}} } \left\| \left(|\xi_2|^{\frac{2-\A}{r}} |\hat{\phi}(\xi_2)| \right)^{\frac{r}{r-2}}  \right\|_{L^{(r-2)'  } } \\
&=&C\left\|  \left(|\cdot|^{\frac{2-\A}{r}} |\hat{\phi}(\cdot)| \right)^{\frac{r}{r-2}} \right\|_{L^{ \frac{r-2}{r-3}} }^2 \quad \text{note that}\,(r-2)'=\frac{r-2}{r-3} \\
&=& C\left\| |\cdot|^{ \frac{2-\A}{r}} \hat{\phi} \right\|^{\frac{2r}{r-2} }_{L^{\frac{r}{r-3}}}.
\end{eqnarray*}
Hence
\begin{equation*}
\|K_1\|_{L^{ \frac{r}{2}} }^{ \frac{r}{r-2} } =\|G\|_{L^{\frac{r}{r-2}}_{\eta_1\eta_2}}^{\frac{r}{r-2}} \le C \left\| |\cdot|^{ \frac{2-\A}{r}} \hat{\phi} \right\|^{\frac{2r}{r-2} }_{L^{\frac{r}{r-3}}}.
\end{equation*}
Since $K_2,K_3,K_4$ have the same upper bound, we have
\begin{eqnarray*}
\| \U \phi\|_{L^r_{xt}}^2 &=& \| |\U \phi|^2 \|_{L^{ \frac{r}{2}} } \le \|K_1 \|_{L^{ \frac{r}{2}} }+\|K_2 \|_{L^{ \frac{r}{2}} }+\|K_3 \|_{L^{ \frac{r}{2}} }+\|K_4 \|_{L^{ \frac{r}{2}} }\\
&\le & C\| |\cdot|^{\frac{2-\A}{r}} \hat{\phi} \|_{L^{\frac{r}{r-3}  }  }^2 \\
&=& C\||D_x|^{\frac{2-\A}{r}} \phi\|_{\hat{L}^{\frac{r}{3}}}^2.
\end{eqnarray*}
Finally, putting $r=3p$, which is possible if $p>4/3$, yields the desired estimate when $1<\A\le 2$.

For the case where $0<\A<1$, we use (\ref{Jest2}) to obtain 
\begin{equation*}
\left| \xi_1^{\A-1}-\xi_2^{\A-1} \right|^{-\frac{2}{r-2}}  
\le C\frac{\xi_1^{\frac{2}{r-2}} \xi_2^{\frac{2(1-\A)}{r-2}}+\xi_2^{\frac{2}{r-2}} \xi_1^{\frac{2(1-\A)}{r-2}} }{|\xi_1-\xi_2|^{\frac{2}{r-2}}}.
\end{equation*}

Thus we have
\begin{eqnarray*}
\|G\|_{L^{\frac{r}{r-2}}_{\eta_1\eta_2}}^{\frac{r}{r-2}}
&\le &
C\iint_{\R^2} \frac{|\xi_1|^{\frac{2}{r-2} }  |\xi_2|^{\frac{2(1-\A)}{r-2}} }{|\xi_1 -\xi_2 |^{\frac{2}{r-2}}}
\left| \hat{\phi}(\xi_1)
\overline{\hat{\phi}(\xi_2)} \right|^{\frac{r}{r-2}}d\xi_1d\xi_2\\
&& +\iint_{\R^2} \frac{|\xi_2|^{\frac{2}{r-2} }  |\xi_1|^{\frac{2(1-\A)}{r-2}} }{|\xi_1 -\xi_2 |^{\frac{2}{r-2}}}
\left| \hat{\phi}(\xi_1)
\overline{\hat{\phi}(\xi_2)} \right|^{\frac{r}{r-2}} d\xi_1d\xi_2.
\end{eqnarray*}

We rewrite the first term of the right hand side as
\begin{equation*}
\iint_{\R^2} \frac{  \left(|\xi_1|^{\frac{2-\A}{r}} |\hat{\phi}(\xi_1)| \right)^{\frac{r}{r-2}} \left(|\xi_2|^{\frac{2-\A}{r}} |\hat{\phi}(\xi_2)| \right)^{\frac{r}{r-2}}   } {   |\xi_1|^{-\frac{\A}{r-2}}  |\xi_1-\xi_2|^{\frac{2}{r-2}} |\xi_2|^{\frac{\A}{r-2}    }  }  d\xi_1 d\xi_2.
\end{equation*}
We then apply Proposition \ref{WHLS} with
\begin{equation*}
k_1=-\frac{\A}{r-2},\quad k_2=\frac{\A}{r-2}
\end{equation*}
and with $f_1,f_2,q_1,q_2,\lambda$ the same as above and arguing as in the case of $\A>1$, we see that
it is bounded by above by  $C\| |D|^{\frac{2-\A}{r}} \phi \|_{\hat{L}^{r/3}}^{\frac{2}{r-2}}$.  The estimate for the second term is similar. Consequently,
the desired estimate is proved for $0<\A<1$.    
\end{proof}


\begin{rem} \label{FSremark2}
It seems both estimates in Lemma \ref{JestL} are required to prove Theorem \ref{FS}.  In fact, if we employ (\ref{Jest1}), we then
apply Proposition \ref{WHLS} with $k_2=\frac{2-\A}{r-2}$, but this violate the condition $k_2<q_2$ on the power of the weight for $0<\A<1$.  Similarly,
we set $k_2=\frac{\A}{r-2}$ if we use (\ref{Jest2}), but condition $k_2<q_2$ dose not hold if $\A>1$.  Note also that
(\ref{Jest1}) is stronger than (\ref{Jest2}) if $\beta > 0$, while (\ref{Jest2}) is stronger than (\ref{Jest1}) if $\beta < 0$.
To see this, we just apply Young's inequality to the right hande side of (\ref{Jest1}) and (\ref{Jest2}).  Let $(i,j)=(1,2)$ or $(2,1)$.  Then we have
\begin{eqnarray*}
\xi_i^{1-\beta} &=&(\xi_i^{\frac{ 1}{ 1+\beta}} \xi_j^{-\frac{\beta }{1+\beta }} )\cdot (\xi_i^{-\frac{\beta^2}{1+\beta}} \xi_j^{\frac{\beta}{1+\beta}} )
\simleq (\xi_i^{\frac{ 1}{ 1+\beta}} \xi_j^{-\frac{\beta }{1+\beta }} )^{1+\beta} +(\xi_i^{-\frac{\beta^2}{1+\beta}} \xi_j^{\frac{\beta}{1+\beta}} )^{\frac{1+\beta}{\beta}} \\
&=&  \xi_i \xi_j^{-\beta} +\xi_j \xi_i^{-\beta},
\end{eqnarray*}
which is valid if $\beta >0$.  On the other hand, we have
\begin{equation*}
\xi_i \xi_j^{-\beta} \simleq \xi_i^{1-\beta} +(\xi_j^{-\beta})^{\frac{1-\beta}{-\beta}}=\xi_i^{1-\beta} +\xi_j^{1-\beta},
\end{equation*}
which is valid if $\beta<0$.

Finally, it is not difficult to see that $\A\le 2$ is not required to apply Hardy--Littlewood--Sobolev type inequality in the proof of the case $1<\A<2$.  In fact, the condition can be relaxed to $\A<3$.  However, we discarded the higher order case here since we want to focus only on the case $\A\le 2$ and the other cases shall be treated in our next paper.
\end{rem}

\subsection{Smoothing and Maximal Function Estimates and Proof of Corollary \ref{coroffGFS} }\label{FScorproof}

Corollary \ref{coroffGFS} can be deduced by interpolating \eqref{eqFS} and smoothing and maximal function estimates to be established below (see also Figure \ref{fig:acceptableintro} in the previous section).  We begin with a smoothing estimate.
\begin{lem} \label{SmoothingL2}
Let $0<\A\le 2,\,\A\neq 1$ and let $\phi \in L^2$.  Then
\begin{equation}
    \left\||D_x|^{\frac{\A-1}{2}} U_{\A}(t) \phi  \right\|_{L^{\infty}_x L_t^2 (\R)} \le  C\|\phi\|_{L^2}. \label{smoothing}
\end{equation}

\end{lem}

\begin{proof}  The estimate is contained in \cite[Theorem 4.1]{Kenig_1991_OsciIntDisp}, where the authors prove a similar estimate for more general dispersive equations. Here we give a direct proof, following \cite[Theorem 4.3]{LinaresPonce_2015_DisperBook}, where \eqref{smoothing} for $\A=2$ is proved.  We write $\phi=\phi_{+}+\phi_{-}$, where $\hat{\phi_{\pm}}:=
\mathbf{1}_{\R_{\pm}} \hat{\phi}$.  Let $\sigma \in \R$.  We have for a fixed $x\in \R$,
\begin{align*}
    \left\| |D_x|^{\sigma} U_{\A}(t) \phi_{+} \right\|_{L_t^2(\R)}^2 &= \int^{\infty}_{-\infty}
    \left| \int^{\infty}_{-\infty}
    e^{ix\xi} |\xi|^{\sigma} e^{-it|\xi|^{\A}} 
    \hat{\phi}_{+} (\xi) d \xi\right|^2 dt\\
    &= \int^{\infty}_{-\infty}
    \left| \int^{\infty}_0
    e^{ix\xi} \xi^{\sigma} e^{-it\xi^{\A}} 
    \hat{\phi}_{+} (\xi) d \xi\right|^2 dt.
\end{align*}
By the change of variables $r=\xi^{\A}$, the right hand side can be written as
\begin{align*}
  \int^{\infty}_{-\infty}
    \left| \int^{\infty}_0
    e^{ixr^{\frac{1}{\A}}} r^{\frac{\sigma}{\A}}
    e^{-itr} 
    \hat{\phi}_{+} (r^{\frac{1}{\A}}) r^{\frac{1-\A}{\A}}d r\right|^2 dt
    &= \int^{\infty}_{-\infty} 
    \left| \mathcal{F}_{r\to t} 
    \left[ e^{ix|r|^{\frac{1}{\A}}} |r|^{\frac{\sigma+1-\A }{ \A}   } \hat{\phi}_{+} (|r|^{\frac{1}{\A}}) \mathbf{1}_{r\ge 0}(r) \right] \right|^2 dt \\
    &=
     \int^{\infty}_{-\infty} 
     \left|  e^{ix|r|^{\frac{1}{\A}}} |r|^{\frac{\sigma+1-\A }{ \A}   } \hat{\phi}_{+} (r^{\frac{1}{\A}})\right|^2 \mathbf{1}_{r\ge 0}(r)dr\\
     &=\int^{\infty}_0 
     \left|  e^{ixr^{\frac{1}{\A}}} r^{\frac{\sigma+1-\A }{ \A}   } \hat{\phi}_{+} (r^{\frac{1}{\A}})\right|^2 dr,
\end{align*}
where $\mathcal{F}_{r\to t}$ denotes the Fourier transform from functions of $r$ to functions of $t$ and the Plancherel identity was used in the last
equality.  Coming back to the original variable
the right hand side is equal to
\begin{equation*}
    c\int^{\infty}_0 |\xi^{\sigma+1-\A} 
    \hat{\phi}_{+} (\xi) |^2 \xi^{\A-1} d\xi=
    c \int^{\infty}_{-\infty} |\xi|^{2\sigma+1-\A} 
    |\hat{\phi}(\xi) |^2 d\xi.
\end{equation*}
If we take $\sigma=\frac{\A-1}{2}$, this equals to
$c\|\hat{\phi}\|_{L^2}^2=c\|\phi\|_{L^2}^2$.
We get a similar identity for $\phi_{-}$.  Consequently, we obtain (\ref{smoothing}).
\end{proof}

\begin{prop}  \label{maximalfunction}
Let $\A \in (0,\infty) \setminus \{ 1\}$. Then, the estimate
\begin{equation}
    \left\||D_x|^{-\frac{1}{4}} U_{\A}(t) \phi \right\|_{L_x^4 L_t^{\infty} (\R)} \le C\|\phi\|_{L^2}. \label{maximal}
\end{equation}
holds true.

\end{prop}

\begin{proof}
 The estimate is also obtained as a special case of \cite[Theorem 2.5]{Kenig_1991_OsciIntDisp}, where the authors proved a similar maximal function estimate for the solutions to more general dispersive equations.  As well as the previous lemma, we give a direct proof for convenience.  The estimate follows immediately from an estimate for oscillatory integrals established by Zhang \cite{Zhang_2014_PointCgt}.  We have (\cite[Lemma 2.2]{Zhang_2014_PointCgt}): For $0<\beta<1$ and $\A>0,\,\A\neq 1$, 
\begin{equation}
\left|\int_{\R} e^{ix\xi +it|\xi|^{\A}}  d\xi\right| \le C (|t|^{\frac{2\beta -1}{2(\A-1)}} |x|^{-\frac{1}{2}-\frac{2\beta -1}{2(\A-1)}}+|x|^{\beta-1}) \quad \forall x\in \R. \label{oscillatoryint}
\end{equation}

We want to show that
\begin{equation}
    \left\|\int_{\R}
    e^{ix\xi-it(x)|\xi|^{\A}} \hat{\phi} (\xi) d\xi \right\|_{L^4} \le C \left\| 
    |D_x|^{\frac{1}{4} }\phi \right\|_{L^2},
    \label{maximalproof}
\end{equation}
for all measurable functions $t :\R\to \R$.  By duality, the left-hand side can be rewritten as
\begin{equation}
\sup_{f\in L^{4/3},\,\|f \|_{L^{4/3}}=1} \left| 
\int_{\R} \overline{f(x)} \int_{\R} e^{ix\xi -it(x)|\xi|^{\A}}\hat{\phi} (\xi) d\xi dx \right|.
\end{equation}
By the Cauchy--Schwartz inequality, we have  
\begin{align*}
\left| 
\int_{\R} \overline{f(x)} \int_{\R} e^{ix\xi -it(x)|\xi|^{\A}}\hat{\phi} (\xi) d\xi dx \right| =&\left| 
\int_{\R} \hat{\phi} (\xi)  \int_{\R} e^{ix\xi -it(x)|\xi|^{\A}}\overline{f(x)} dx d\xi  \right| \\
\le \biggl( \int &|\xi|^{\frac{1}{2}} |\hat{\phi}(\xi)|^2 d\xi  \biggr)^{\frac{1}{2}}
\left(\int_{\R} \frac{1}{|\xi|^{\frac{1}{2}}} \left|  
\int_{\R} e^{ix\xi -it(x)|\xi|^{\A}}\overline{f(x)} dx\right|^2 d\xi \right)^{\frac{1}{2}}\\
\le \left\| |D_x|^{\frac{1}{4}} \phi \right\|_{L^2} &\left( \int_{\R} \frac{1}{|\xi|^{\frac{1}{2}}} 
\int_{\R} \int_{\R} e^{i(x-y)\xi -i(t(x)-t(y))|\xi|^{\A}} f(x) \overline{f(y)} dxdy d\xi \right)^{\frac{1}{2}}.
\end{align*}
Now applying \eqref{oscillatoryint} with $\beta=1/2$ gives us
\begin{align*}
 \int_{\R} \frac{1}{|\xi|^{\frac{1}{2}}} 
\int_{\R} \int_{\R} e^{i(x-y)\xi -i(t(x)-t(y))|\xi|^{\A}} f(x) \overline{f(y)} dxdy d\xi  \\
=
\int_{\R} \int_{\R} \biggl(\int_{\R}
& e^{i(x-y)\xi -i(t(x)-t(y))|\xi|^{\A}} \frac{d\xi}{|\xi|^{\frac{1}{2}} } \biggr) f(x) \overline{f(y)} dxdy \\
&\le \int_{\R} \int_{\R} \frac{f(x)\overline{f(y)} }{|x-y|^{\frac{1}{2}}} dxdy.
\end{align*}
By the Hardy--Littlewood--Sobolev inequality, the right-hand side is bounded by above by
$\|f\|_{L^{4/3}}^2=1 $.  This conludes the proof of \eqref{maximalproof}. 
\end{proof}

\section{Key Ingredients}\label{preliminaries}

\subsection{Linear and Nonlinear Estimates}
Estimates of the following type for solutions to the inhomogeneous equation are established by Masaki and Segata \cite{Masaki_2016_FLpKdV} in the case of the Airy equation.
\begin{prop} \label{inhomoest}
Let $4/3<p<4$.  Assume that $(q_1,r_1,\sigma_1)$ is $(\A,p)$-acceptable and $(q_2,r_2,\sigma_2)$ is $(\A,p')$-acceptable.  
Then
\begin{equation}
    \left\| \int^t_0 U_{\A}(t-\tau) F(\tau) d\tau    \right\|_{L_t^{\infty} (I ; \Lp)}
    \le C \left\|  |D_x|^{-\sigma_2} F\right\|_{L_x^{r_2'} L^{q_2'}(I) }
    \label{GFSdualprop}
\end{equation}
and
\begin{equation}
    \left\||D_x|^{\sigma_1} \int^t_0  U_{\A}(t-\tau) F(\tau) d\tau    \right\|_{L_x^{r_1}L^{q_1}_t(I) } 
    \le C \left\|  |D_x|^{-\sigma_2} F\right\|_{L_x^{r_2'} L^{q_2'}(I) } . \label{inhomoestprop}
\end{equation}
    \end{prop}

    \begin{proof} \eqref{GFSdualprop} follows from Corollary \ref{coroffGFS} via the standard duality argument.  Indeed, writing
    \begin{align*}
    \left\|\int_I U_\A(-\tau) F(\tau )d \tau \right\|_{\Lp}&=\sup_{f: \|f\|_{\widehat{L^{p'}}}=1}
    \int_{\R} \int_I
    U_{\A}(-\tau) F(\tau)d\tau \overline{f(x)}dx\\
    &=\sup_{f: \|f\|_{\widehat{L^{p'}}}=1} \int_{\R} \int_I |\xi|^{-\sigma_2}e^{i\tau|\xi|^{\A}} \hat{F} (\tau,\xi) |\xi|^{\sigma_2}\overline{\hat{f}(\xi)} d\tau d\xi\\
 &=\sup_{f: \|f\|_{\widehat{L^{p'}}}=1} \int_{\R} \int_I  |D_x|^{-\sigma_2}F(\tau,x)  \overline{(|D|^{\sigma_2}U_{\A}(\tau)f)(x)} d\tau dx\\
 &\le \sup_{f: \|f\|_{\widehat{L^{p'}}}=1} \||D_x|^{-\sigma_2} F \|_{L^{r_2'}_xL^{q_2'}_t(I)}
 \||D_x|^{\sigma_2} U_{\A}(t) f \|_{L^{r_2}_x L^{q_2}_t(I)}
    \end{align*}
and applying \eqref{OffdGFS} to the right-hand side yields the desired estimate.

    \eqref{inhomoestprop} can be obtained by arguing as in the proof of \cite[Proposition 2.5]{Masaki_2016_FLpKdV}, where a similar estimate for the Airy equation is proved using the Christ--Kiselev lemma \cite{christ_2001_maximal,Molinet_2004_BeOno}.

    \end{proof}

\begin{cor}
Let $s\in \R$ and $4/3<p<4$.  Assume that $(q_1,r_1,\sigma_1)$ is $(\A,p)$-acceptable and $(q_2,r_2,\sigma_2)$ is $(\A,p')$-acceptable.  
Then
\begin{equation}
     \left\| |D_x|^{s+\sigma_1} U_{\A}(t) f \right\|_{L_x^{r_1}(L_t^{q_1})} 
     \le C\|f\|_{\HHSP},  \label{Key1}
 \end{equation}

\begin{equation}
    \left\| \int^t_0 U_{\A}(t-\tau) F(\tau) d\tau    \right\|_{L_t^{\infty} (I ; \HHSP)}
    \le C \left\|  |D_x|^{s-\sigma_2} F\right\|_{L_x^{r_2'} L^{q_2'}(I) } \label{Key2}
\end{equation}
and
\begin{equation}
    \left\||D_x|^{s+\sigma_1} \int^t_0  U_{\A}(t-\tau) F(\tau) d\tau    \right\|_{L_x^{r_1}L^{q_1}_t(I) } 
    \le C \left\|  |D_x|^{s-\sigma_2} F\right\|_{L_x^{r_2'} L^{q_2'}(I) } . \label{key3}
\end{equation}

\end{cor}

We will need two further lemmas from \cite{Masaki_2016_FLpKdV}.

We introduce a Lipschitz $\mu-$norm ($\mu>0$) as follows.

\begin{defn}
Let $\mu =N+\beta$ with $N\in \N \cup \{0\},\,0< \beta \le 1$.  For $G:\C\to \C$, we define
\begin{equation*}
\|G\|_{{\rm Lip}\,\mu}:=
\sum_{j=0}^N \sup_{z \in \C\setminus \{0\}}
\frac{|G^{(j)}(z)|}{|z|^{\mu-j}}+\sup_{x\neq y}
\frac{|G^{(N)}(x)-G^{(N)}(y) | }{|x-y|^{\beta}}
\end{equation*}
  where $G^{(j)}$  is the $j-$th derivative of $G.$ We say $G\in  \rm Lip  \ \mu$ if $G \in C^N(\mathbb R)$ and $\|G\|_{\rm Lip \ \mu} < \infty.$
\end{defn}

\begin{lem} (\cite[Lemma 3.6]{Masaki_2016_FLpKdV})\label{Leibniz}
Let $I\subset \R$.  Let $s\ge 0$ and let 
$p,q,p_i,q_i \in (1,\infty)$ be such that
\begin{equation*}
    \frac{1}{p}=\frac{1}{p_1}+\frac{1}{p_2}
    =\frac{1}{p_3}+\frac{1}{p_4},\quad
    \frac{1}{q}=\frac{1}{q_1}+\frac{1}{q_2}=\frac{1}{q_3}+\frac{1}{q_4}.
\end{equation*}
Then
\begin{equation*}
    \left\| |D_x|^s (fg) \right\|_{L_x^p L_t^q (I)} \le C \left( 
    \left\| |D_x|^s f \right\|_{L_x^{p_1} L_t^{q_1} (I)}\left\| g \right\|_{L_x^{p_2} L_t^{q_2} (I)} +\left\| f\right\|_{L_x^{p_3} L_t^{q_3} (I)}\left\| |D_x|^s g \right\|_{L_x^p L_t^q (I)}\right).
\end{equation*}

\end{lem}

\begin{lem}\label{Lipe}
    Let $I\subset \R,\,\mu >1$ and $s\in (0,\mu)$.  Let $q,p_1,p_2,q_1,q_2 \in (1,\infty)$   be such that
    \begin{equation*}
        \frac{1}{p}=\frac{\mu-1}{p_1}+\frac{1}{p_2},\quad \frac{1}{q}=\frac{\mu-1}{q_1}+\frac{1}{q_2}.
    \end{equation*}
    Then for any $G\in {\rm Lip}\,\mu$, 
    \begin{equation*}
    \left\||D_x|^s G(f) \right\|_{L^p_x L^q_t(I)}
    \le C \|G\|_{{\rm Lip}\,\mu}
    \|f \|^{\mu-1}_{L^
    {p_1}_x L^{q_1}_t(I)} 
    \left\| |D_x|^s f\right\|_{L^{p_2}_x L^{q_2}_t(I)  }.
    \end{equation*}
\end{lem}

Finally, we present a Sobolev-type inequality for the $L_x^r L_t^q$ spaces.

\begin{lem}
Let $I\subset \R$ and let $1\le q \le \infty$ and $1<r_1<r_2<\infty$.  Then
\begin{equation}
    \|  f \|_{L_x^{r_2}L_t^q (I)} \le 
    C\| |D_x|^s f\|_{L_x^{r_1}L_t^q (I)},\label{Sobolev}
\end{equation}
where $s=1/r_1-1/r_2$.
\end{lem}

\begin{proof}
This is essentially obtained as a special case of \cite[Lemma 3.15]{kenig_1993_contraction}.  Here we give a proof for clarity since the lemma in \cite{kenig_1993_contraction} is proved for some special exponents associated with the Korteweg-de Vries equation.  We write $f=|D_x|^{-s} |D_x|^s f$.  Then we represent $|D_x|^{-s}$ in terms of a convolution (see e.g.\cite[Subsection 1.2.1]{Grafakos_2014_Modern} ) to obtain
\begin{equation}
f(t,x)=c\int_{\R} \frac{(|D_x|^s f) (t,y)}{|x-y|^{1-s}} dy.
\end{equation}
By Minkowski's inequality we have 
\begin{equation*}
\| f(\cdot , x) \|_{L_t^q(I)} \le c
\int_{\R} \frac{ \|(|D_x|^s f)(\cdot, y)\|_{L^q(I)}}{|x-y|^{1-s}} dy.
\end{equation*}
Then the desired estimate follows from the Hardy--Littlewood--Sobolev inequality.
\end{proof}
\subsection{Special exponents and $(\alpha, p)-$acceptibility} \label{Spexponent}
Before we start the proof of Theorem \ref{LWPsubcriticalp}, we introduce several constants and prove their properties.  Let $s_0 \ge 0$.  We define $q_j,r_j,j=0,1,2 $ and $\delta$ as follows.
\begin{align*}
\frac{1}{q_0}&=\frac{1-2sp}{p(2\A-1)},\quad 
\frac{1}{r_0}=\frac{\A-1+sp}{p(2\A-1)}\\
\frac{1}{q_1}&=\frac{1-2sp+2s_0p}{p(2\A-1)},
\quad 
\frac{1}{r_1}=\frac{\A-1+sp-s_0p}{p(2\A-1)},\\
    \delta&=\frac{(\A-1+sp)(\rho-1)}{2\A p}-\frac{1}{2},\\
 \frac{1}{q_2}&=\frac{p-1+2p(s-s_0)+2\delta p}{p(2\A-1)},\quad \frac{1}{r_2}=\frac{(\A-1)(p-1)-p(s-s_0)-\delta p}{p(2\A-1)}.
\end{align*}

Let us explain briefly about these exponents and numbers.  The exponents $q_j,r_j,j=0,1,2,\,s_0$, and $\delta$ appear in the estimate of the integral equation equivalent to \eqref{FNLS}.  Indeed, under suitable conditions, $(q_0,r_0,-s)$ and $(q_1,r_1,s_0-s)$ are $(\A,p)$-acceptable, and $(q_2,r_2,s-s_0+\delta)$ is $(\A,p')$-acceptable, which will enable us to estimate the solution to the integral equation for data $\phi \in \HHSP$ in spaces based on $L_x^{r_0}L_t^{q_0}$ and $L_x^{r_1} L_t^{q_1}$.  The number $s_0$ is a regularity we need to establish a solution in $L_x^{r_1} L_t^{q_1}$.  Finally, we need $\delta$ to modify the exponents so that our functional framework is adapted to the scaling subcritical settings.  In fact, we can take $\delta=0$ in the critical case $s=s_c=s_l$.

\begin{lem} \label{exponentlemma}
Let $1<\A\le 2$.  Assume that $\A,p,\rho$ satisfy \eqref{rhoprange} and $s \in [s_l,s_u)$.  Then 
\begin{enumerate}
\item
$(q_0,r_0,-s)$ is $(\A,p)$-acceptable.

\item  $(q_1,r_1,s_0-s)$ is $(\A,p)$-acceptable if
  \begin{equation*}
 \frac{2\A}{2\A-1}<p\le 2,\quad \text{and}\quad 0\le s_0 <s+\frac{2\A p-p-2\A}{2p} ,
  \end{equation*}
or 
 \begin{equation*}
2\le p <\min (4\A-2,4) \quad \text{and} \quad  0\le s_0 <s+\frac{\A-1}{p} ,
  \end{equation*}
\item 
$(q_2,r_2,s-s_0+\delta)$ is $(\A,p')$-acceptable if
  \begin{equation*}
  \frac{2\A}{2\A-1}<p\le 2,\quad \text{and}\quad s-\frac{(\A-1)(p-1)}{p}+\delta\le s_0 <s+\frac{p-1}{2p}+\delta ,
  \end{equation*}
or 
 \begin{equation*}
 2\le p <\min (4\A-2,4) \quad \text{and} \quad  s-\frac{2\A-p}{2p}+\delta  \le s_0 <s+\frac{3p -2\A p+4\A-4}{4p}+\delta,
  \end{equation*}

\item $\displaystyle{0 \le \delta <
\min \left(\frac{\A-1}{2\A},\,\frac{4\A p-3p-4\A+2}{4\A p} \right)}$.

\item $\delta<r_2^{-1}$ if and only if
\begin{equation*}
s_0 >s-\frac{(\A-1)(p-1)}{p}+2\A \delta.
\end{equation*}
\item 
There exists an $s_0$ such that $(q_1,r_1,s_0-s)$ is $(\A,p)$-acceptable, \,$(q_2,r_2,s-s_0+\delta)$ is $(\A,p')$-acceptable, and $\delta<r_2^{-1}$.
\end{enumerate}
\end{lem}

\begin{proof}
(i) We first prove that $(q_0,r_0,-s)$ is $(\A,p)$-acceptable.  It is easy to see that
the triplet satisfies the scaling conditions
\begin{equation*}
\frac{1}{q_0}+\frac{2}{r_0}=\frac{1}{p},\quad -s=\frac{\A}{q_0}+\frac{1}{r_0}-\frac{1}{p}.
\end{equation*}
It remains to check that $q_0> 2,\,r_0>4,\,$ and $1/q_0+1/r_0<1/2$. 

$q_0>2$ is equivalent to 
\begin{equation*}
    \frac{p-2\A p+2}{4p} < s <\frac{1}{2p}.
\end{equation*}

    Note that this inequality holds true since $p-2\A p +2<0$ if $ p>\frac{2\A}{2\A-1}$   It is also easy to see that $r_0>4$ is equivalent to
    \begin{equation*}
        s<\frac{2\A p-p-4\A+4}{4p},
    \end{equation*}   
    which follows from $s<s_u$.
    
Finally,
\begin{equation*}
\frac{1}{q_0}+\frac{1}{r_0}=\frac{\A-sp}{p(2\A-1)} <\frac{1}{2}
\end{equation*}
if and only if
\begin{equation*}
s>\frac{2\A-2\A p+p}{2p},
\end{equation*}
which holds true, since $p>\frac{2\A}{2\A-1}$ implies $2\A-2\A p +p<0$.

Therefore, $(q_0,r_0,-s)$ is $(\A,p)$-acceptable.

We prove (ii).  It is easy to check that the triplet $(q,r,\sigma):=(q_1,r_1,s_0-s)$ satisfies  the scaling conditions \eqref{scaling} and \eqref{dgain} for $(\A,p)$-acceptablity.  It is also not difficult to deduce
\begin{align*}
q_1>2& \quad \text{if and only if} \quad s-\frac{1}{2p} <s_0<s+\frac{2\A p -p-2}{4p} (:=\bar{\kappa}_1)\\
r_1>4 & \quad \text{if and only if} \quad 
s-\frac{2\A p -p-4\A+4}{4p}<s_0<s+\frac{\A-1}{p} (:=\bar{\kappa}_2) \\
\frac{1}{q_1}+\frac{1}{r_1}&<\frac{1}{2} \quad \text{if and only if} \quad s_0 <s+\frac{2\A p-p-2\A}{2p} (:=\bar{\kappa}_3).
\end{align*}
By the assumption $s<s_u$, we see that
\begin{equation*}
\max \left( s-\frac{1}{2p},\, 
s-\frac{2\A p -p-4\A+4}{4p} \right) <0.
\end{equation*}

 For $\bar{\kappa}_1,\bar{\kappa}_2,\bar{\kappa}_3$, a computation shows that
\begin{equation*}
    \bar{\kappa}_1-\bar{\kappa}_2 =\frac{\bar{\kappa_3}-\bar{\kappa_2}}{2}=\bar{\kappa}_3-\bar{\kappa}_1=\frac{(2\A-1)(p-2)}{4p}.
\end{equation*}

Thus we have
\begin{equation*}
    \min(\bar{\kappa_1} ,\bar{\kappa_2},\bar{\kappa_3} )
    =\begin{cases}
        \bar{\kappa_3} & \text{if} \quad p\le 2 \\
        \bar{\kappa_2} & \text{if} \quad p\ge 2
 \end{cases}
\end{equation*}
and we see that $(q_1,r_1,s_0-s)$ is $(\A,p)$-acceptable if
\begin{equation*}
    0\le s_0 <\begin{cases}
        \bar{\kappa}_3 & \text{if} \quad p\le 2 \\
        \bar{\kappa}_2 & \text{if} \quad p\ge 2.

    \end{cases}
\end{equation*}
This proves (ii).  The proof of (iii) is similar.  A simple computation shows that $(q,r,\sigma):=(q_2,r_2,s-s_0+\delta)$ satisfies the scaling conditions for the $(\A,p')$-acceptability.  For the other conditions, we see that

\begin{align*}
q_2>2& \quad \text{if and only if} \quad s+\frac{-2\A p +3p-2}{4p} +\delta \,(:=\underline{\kappa}_1) <s_0<s+\frac{p-1}{2p}+\delta (:=\bar{\kappa}_4)\\
r_2>4 & \,\, \text{if and only if} \,\, s-\frac{(\A-1)(p-1)}{p}+\delta (:=\underline{\kappa}_2) <s_0<s+\frac{3p -2\A p+4\A-4}{4p}+\delta\, (:=\bar{\kappa}_5) \\
\frac{1}{q_2}+\frac{1}{r_2}&<\frac{1}{2} \quad \text{if and only if} \quad s-\frac{2\A-p}{2p}+\delta  (:=\underline{\kappa}_3)<s_0.
\end{align*}

Now we have
\begin{align*}
    \underline{\kappa}_1 -\underline{\kappa}_2&=\frac{\underline{\kappa}_3-\underline{\kappa}_2}{2}=\underline{\kappa}_3-\underline{\kappa}_1 =\overline{\kappa}_4-\overline{\kappa}_5=\frac{(2\A-1)(p-2)}{4p}
\end{align*}
from which we see that $(q_2,r_2,s-s_0+\delta)$ is $(\A,p')$-acceptable if
\begin{align}
    \underline{\kappa}_2  <s_0 <\overline{\kappa}_4 \quad \text{when} \quad p \le 2 \\
    \underline{\kappa}_3  <s_0 <\overline{\kappa}_5 \quad \text{when} \quad p \ge 2.
\end{align}

We prove (iv).  $\delta\ge 0$ follows immediately from the assumption $s\ge s_l$.  The inequality  $\delta <\frac{\A-1}{2\A}$ is equivalent to
\begin{equation}
    (\A-1+sp)(\rho-1)-\A p <p (\A-1).
\end{equation}
It is easy to verify this inequality.  Indeed, since $s<1/2p$ and $\rho \le 2p +1$, we have
\begin{equation*}
(\A-1+sp)(\rho-1)-\A p <(\A-1/2)\cdot 2p-\A p=p(\A-1).
\end{equation*}
It remains to prove
\begin{equation}
    \delta < \frac{4\A p-3p-4\A+2}{4\A p}.
\end{equation}
It is easy to check that this is equivalent to 
\begin{equation}
    s< \frac{6\A p-3p-4\A+2}{2p(\rho-1)}-\frac{\A-1}{p}.
\end{equation}
Thus, the inequality follows from the assumption $s<s_u$.

(v) is straightforward.

Finally, we prove (vi).  In view of (ii)(iii)(iv)(v), ti is enough to show that 
\begin{equation}
\max \left( 0,\, s-\frac{(\A-1)(p-1)}{p}+2\A\delta \right) 
<
\min \left( s+\frac{2\A p-p-2\A}{2p},\, s+\frac{p-1}{2p}+\delta \right) \label{s0range1}
\end{equation}
when $p\le  2$ and
\begin{equation}
\begin{aligned}
\max &\left( 0, s-\frac{(\A-1)(p-1)}{p}+2\A\delta,s-\frac{2\A-p}{2p}+\delta \right) \\
<s+\min &\left( \frac{\A-1}{p}, \frac{3p -2\A p +4\A-4}{4p}+\delta\right) \label{s0range2}
\end{aligned}
\end{equation}
when $p\ge 2$.  We first consider the case $p\le 2$.  Clearly, $2\A p-p-2\A>0$ since $p>\frac{2\A}{2\A-1}$ so the right-hand side is positive.  It remains to check that
\begin{equation*}
 s-\frac{(\A-1)(p-1)}{p}+2\A\delta <s+\frac{p-1}{2p}+\delta,
\end{equation*}
and
\begin{equation*}
 s-\frac{(\A-1)(p-1)}{p}+2\A\delta <s+\frac{2\A p-p-2\A}{2p},
\end{equation*}
The two inequalities follow from (iv).  Indeed,  the second inequality holds if and only if
\begin{equation}
    \delta < \frac{4\A p-3p-4\A+2}{4\A p}. \label{deltacondition}
\end{equation}
The first one is equivalent to $\delta <\frac{p-1}{2p}$.  By (iv), it is enough to check that
\begin{equation*}
\delta < \frac{4\A p-3p-4\A+2}{4\A p}<\frac{p-1}{2p}.
\end{equation*}
The second inequality is equivalent to $p(2\A-3)<2\A-2$, which holds true if $p\le 2$ and $\A\le 2$.
  Thus the desired inequality for $p\le 2$ is proved.  We then consider the case $p \ge 2$.  We first show that the right hand side of the inequality in question is positive.  Since $\rho \le 2p+1, s \ge s_l$, and $\delta \ge 0$, we have  
\begin{align*}
 s+\frac{3p-2\A p+4\A-4}{4p}+\delta &\ge s_l +\frac{3p-2\A p+4\A-4}{4p} \\
 & \ge \frac{\A}{2p}-\frac{\A-1}{p}+\frac{3p-2\A p+4\A-4}{4p} \\
 &=\frac{2\A-(2\A-3)p}{4p}.
\end{align*}
The last term is positive if and only if $p(2\A-3)<2\A$ which holds if $\A\le 3/2$.  When $\A>3/2$, the conditon is equivalent to $p<\frac{2\A}{2\A-3}$ which is satisfied since $\A\le 2$ and $p<4$.
\begin{equation*}
    s+\frac{3\A-2\A p+4\A-4}{4p}\ge s_c+\frac{3\A-2\A  p+4\A-4}{4p} =\frac{5\A-2}{4p}>0.
\end{equation*}
Thus, we see that the right hand side of the wanted inequality is positive.  Finally, we compare the other factors of the inequality.  It is not difficult to check that
\begin{align*}
&s-\frac{(\A-1)(p-1)}{p}+2\A\delta
<s+\frac{\A-1}{p}\quad \text{if and only if} \quad \delta<\frac{\A-1}{2\A} \\
&s-\frac{(\A-1)(p-1)}{p}+2\A\delta
< s+\frac{3p -2\A p +4\A-4}{4p}+\delta \quad \text{if and only if} \quad \delta<\frac{1}{4} \\
&s-\frac{2\A-p}{2p}+\delta
<s+\frac{\A-1}{p}\quad \text{if and only if} \quad s<\frac{2\A(2\A-1)}{p}-\frac{\A-1}{p}\\
&s-\frac{2\A-p}{2p}+\delta
< s+\frac{3p -2\A p +4\A-4}{4p}+\delta \quad \text{if and only if} \quad p<4.
\end{align*}
Clearly, these inequalities hold by the assumptions.  In particular, for the second inequality, note that $\delta<\frac{\A-1}{2\A} <\frac{1}{4}$, since $\A\le 2$ and the condition for the third inequality holds since $s<s_u$.
\end{proof}

\section{Well-posedness in  Fourier Sobolev spaces $\HHSP$} \label{proofLWP}
In this chapter, we prove the well-posedness results for \eqref{FNLS}.  Basically, we follow the approach by Masaki and Segata \cite{Masaki_2016_FLpKdV}. 
\subsection{Proof of Theorem \ref{LWPsubcriticalp}}
We first prove Theorem \ref{LWPsubcriticalp}.  We fix $s_0$ so that $(q_1,r_1,s_0-s)$ is $(\A,p)$-acceptable, $(q_2,r_2,s-s_0-\delta)$ is $(\A,p')$-acceptable, and $\delta<r_2^{-1}$.  This is possible by Lemma \ref{exponentlemma}.

Let $T>0$ and $I_T:=[0,T]$.  We set

\begin{equation*}
S(T):=\{u \,|\, u \in L^{r_0}_x L_t^{q_0}  (I_T),\quad |D_x|^{s_0}\, u
\in L^{r_1}_x L^{q_1}_t (I_T)\,\},
\end{equation*}
equipped with the norm
\begin{equation*}
\|u\|_{S(T)}:=\| u\|_{L^{r_0}_x L_t^{q_0} (I_T)}+\left\||D_x|^{s_0}u \right\|_{L^{r_1}_x L^{q_1}_t (I_T)}.
\end{equation*}

The proof of Theorem \ref{LWPsubcriticalp} proceeds as follows.  We first construct a solution \eqref{gFNLS} in $S(T)$ via a fixed point argument.  Indeed, using the Fefferman--Stein estimate and inhomongeneous Strichartz estimates in $L_x^rL_t^q$, we can estimate the $S(T)$-norm of the integral equation equivalent to the Cauchy problem.  In particular, the non-linear part of the equation can be estimated in terms the norm based on $\| |D_x^{s_0-\delta} |u|^{\rho-1}u \|_{L^{r_2'}_xL_t^{q_2'}}$.  Then, the non-linear estimates allow us to bound the norm using $\|u\|_{L^{r_0}_x L_t^{q_0}}$ and $\||D_x|^{s_0} u\|_{L_x^{r_1}L_t^{q_1}}$, which leads to the existence of a solution in $S(T)$ for some small enough $T>0$. Finally, we show the persistence property of the solution using the dual of  Fefferman--Stein estimate \eqref{GFSdualprop}.

We first establish a unique solution of \eqref{FNLS} in $S(T)$.  For $M>0$ we set $$\mathcal{\dot{B}}_M:=\left\{ \phi \in \HHSP\,|\, \|\phi \|_{\HHSP} \le M \right\}.$$
\begin{prop}\label{KeyEUthm}
Let $\A,p,s$ be as in Theorem \ref{LWPsubcriticalp}.  Then:
\begin{enumerate}
    \item For any $\phi \in \HHSP$ there exist $T\sim \| \phi \|_{\HHSP}^{-\frac{\A}{s-s_c} }$ and a unique solution $u \in S(T)$ of \eqref{FNLS}.  
    \item For any $M>0$, there exists $T_M>0$ with the following properties: for any $\phi \in \mathcal{\dot{B}}_M$ \eqref{FNLS} has a unique solution $u \in S(T_M)$.  Furthermore, the map data-solution is Lipschitz from $\dot{\mathcal{B}}_M \to S_{T_M}$.  More precisely, there is $C>0$ such that
\begin{equation}
    \|u_1-u_2\|_{S(T_M)}
    \le C \|\phi_1-\phi_2\|_{\HHSP} \label{Cdependenceprop}
\end{equation}
for any two solutions $u_1,u_2 \in S(T_M)$ of \eqref{FNLS} with $u_j(0)=\phi_j \in \mathcal{\dot{B}}_M, j=1,2$.
\end{enumerate}
\end{prop}

\textbf{Proof of Proposition \ref{KeyEUthm}}.
We first prove (i).  We estimate the right-hand side of the corresponding integral equation
\begin{equation*}
    u(t)=U_{\A}(t)\phi +i\int^t_0 U_{\A}(t-\tau) |u(\tau) |^{\rho-1} u(\tau) d\tau.
\end{equation*}
Since $(q_0,r_0,-s)$ and $(q_1,r_1,s_0-s)$ are $(\A,p)$-acceptable, we may apply \eqref{Key1} to obtain
\begin{align*}
    \|U_{\A}(t)\phi \|_{ L^{r_0}_x L_t^{q_0}(I_T) }  &=  \||D_x|^{s+(-s )}U_{\A}(t)
    \phi \|_{ L^{r_0}_x L_t^{q_0} (I_T) } \\
    & \le C \| \phi \|_{\HHSP}
\end{align*}
and
\begin{align*}
     \left\| |D_x|^{s_0} U_{\A}(t) \phi \right\|_{L_x^{r_1}L_t^{q_1}(I_T)} 
    &= \left\| |D_x|^{s+(s_0-s)} U_{\A}(t)\phi \right\|_{L_x^{r_1}L_t^{q_1} (I_T)} \\
     &\le C\|\phi\|_{\HHSP}. 
 \end{align*}
By these estimates, we get
\begin{equation}
    \|U_{\A}(t)\phi \|_{S(T)} \le C\|\phi \|_{\HHSP}.
\end{equation}
 For the nonlinear term, we write
 \begin{align*}
     \left\|\int^t_0 U_{\A}(t-\tau) |u(\tau) |^{\rho-1} u(\tau) d\tau \right\|_{S(T)}
    &=
    \left\|
    |D_x|^{s+(-s) }  \int^t_0 U_{\A}(t-\tau) |u(\tau) |^{\rho-1} u(\tau) d\tau \right\|_{ L^{r_0}_x L^{q_0}_t (I_T) } \\
    &+
    \left\| |D_x|^{ s+(s_0-s)} \int^t_0 U_{\A}(t-\tau) |u(\tau) |^{\rho-1} u(\tau) d\tau \right\|_{ L^{r_1}_x L^{q_1}_t (I_T) }.
 \end{align*}
 Applying (\ref{key3}), we see that
 the two terms on the right-hand side are smaller than
 \begin{equation}
    C \left\|  |D_x|^{s-(s-s_0+\delta)}|u|^{\rho-1} u  \right\|_{L^{r_2^{'} }_x L^{q_2^{'}}_t (I_T)   }=C \left\| |D_x|^{s_0-\delta} |u|^{\rho-1} u  \right\|_{L^{r_2^{'} }_x L^{q_2^{'}}_t (I_T)   },
    \label{nonlinearterm1}
 \end{equation}
 since $(q_2,r_2,s-s_0+\delta)$ is $(\A,p')$-acceptable.  

 Now we estimate \eqref{nonlinearterm1}.  By the Sobolev-type embedding \eqref{Sobolev} and H\"older's inequality with respect to the time variable, we have
\begin{equation*}
     \left\| |D_x|^{s_0-\delta} |u|^{\rho-1} u  \right\|_{L^{r_2^{'} }_x L^{q_2^{'}}_t (\R)   } \le   C\left\| |D_x|^{s_0} |u|^{\rho-1} u  \right\|_{L^{\tilde{r} }_x L^{q_2^{'}}_t (\R)   } \le
     CT^{\frac{(\rho-1)(s-s_c) }{\A } } C\left\| |D_x|^{s_0} |u|^{\rho-1} u  \right\|_{L^{\tilde{r} }_x L^{\tilde{q}}_t (I_T)   }  ,
 \end{equation*}
where
\begin{equation*}
    \frac{1}{r_2'}=\frac{1}{\tilde{r}}-\delta,\quad \frac{1}{q_2'}=\frac{1}{\tilde{q}}+\frac{\rho-1}{\A}(s-s_c).
\end{equation*}
Note that $\tilde{r}>1$ since $r_2<\delta^{-1}$, and thus \eqref{Sobolev} is applicable.

 Now, observe that
\begin{align}
   \frac{1}{\tilde{r}}&= 1-\frac{1}{r_2}+\delta\\
   &=\frac{1}{r_1}+\frac{\rho-1}{r_0},\\
    &=\frac{1}{r_0}+\left(\frac{1}{r_1}+\frac{\rho-2}{r_0} \right) \left(:=\frac{1}{r_0}+\frac{1}{r_3} \right),
\end{align}
and
\begin{align}
    \frac{1}{\tilde{q}}&=1-\frac{1}{q_2}-\frac{\rho-1}{\A}(s-s_c) \\&=\frac{1}{q_1}+\frac{\rho-1}{q_0}\\
    &= \frac{1}{q_0}+\left(\frac{1}{q_1}+\frac{\rho-2}{q_0} \right)\left(:=\frac{1}{q_1}+\frac{1}{q_3}\right).
\end{align}
Thus, applying Lemma \ref{Lipe} with $G(z)=|z|^{\rho-1}z \in \mathrm{Lip}\,\rho$,  we obtain
\begin{align*}
       \left\| |D_x|^{s_0} |u|^{\rho-1} u  \right\|_{L^{\tilde{r} }_x L^{q_2^{'}}_t (I_T)   } 
       & \le T^{\frac{(\rho-1)(s-s_c)}{\A}} 
       \left\| |D_x|^{s_0} |u|^{\rho-1} u  \right\|_{L^{\tilde{r} }_x L^{\tilde{q}}_t (I_T)   } \\
&   \simleq  CT^{\frac{(\rho-1)(s-s_c)}{\A}} \biggl( \left\||D_x|^{s_0} u \right\|_{L^{r_1 }_x L^{q_1}_t (I_T) } \left\| |u|^{\rho-1}  \right\|_{L^{\frac{r_0}{\rho-1}}_x L^{\frac{q_0}{\rho-1} }_t (I_T)} \\
& + \|u \|_{L^{r_0}_x L^{q_0 }_t (I_T)} 
 \left\| |D_x|^{s_0}|u|^{\rho-1} \right\|_{L^{r_3 }_x L^{q_3}_t (I_T)} \biggr)\\
 &   = CT^{\frac{(\rho-1)(s-s_c)}{\A}} \biggl( \left\||D_x|^{s_0} u \right\|_{L^{r_1 }_x L^{q_1}_t (I_T) } \left\| u  \right\|_{L^{r_0}_x L^{q_0 }_t (I_T)}^{\rho-1} \\
& + \|u \|_{L^{r_0}_x L^{q_0 }_t (I_T)} 
 \left\| |D_x|^{s_0}|u|^{\rho-1} \right\|_{L^{r_3 }_x L^{q_3}_t (I_T)} \biggr).
\end{align*}

Next we estimate the second norm of the second term on the right-hand side.  By Lemma \ref{Lipe},  we have
\begin{align*}
 \left\| |D_x|^{s_0}|u|^{\rho-1} \right\|_{L^{r_3 }_x L^{q_3}_t (I_T)}
 &\le \left\| |\cdot|^{\rho-2} \right\|_{{\rm Lip\,} (\rho-2)}
 \|u\|^{\rho-2}_{L^{r_0}_x L^{q_0 }_t (I_T)}
  \left\| |D_x|^{s_0} u\right\|_{L^{r_1 }_x L^{q_1}_t (I_T)  },\\
  &= C \|u\|^{\rho-2}_{L^{r_0}_x L^{q_0 }_t (I_T)}
  \left\| |D_x|^{s_0} u \right\|_{L^{r_1 }_x L^{q_1}_t (I_T)  }.
\end{align*}

Consequently, we see that the right-hand side of \eqref{nonlinearterm1} can be controlled by $\|u\|_{L^{r_0}_x L^{q_0}_t(I_T)}$ and $\||D_x|^{s_0} u \|_{L^{r_1}_x L^{q_1}_t(I_T)}$  and thus we have 
 \begin{align}
     \left\|\int^t_0 U_{\A}(t-\tau) |u(\tau) |^{\rho-1} u(\tau) d\tau \right\|_{S(T)} &\le
    C \left\| |D_x|^{s_0-\delta}|u|^{\rho-1} u  \right\|_{L^{r_2^{'} }_x L^{q_2^{'}}_t (I_T)   } \\
    & \le 
    C T^{\frac{(\rho-1)(s-s_c)}{\A}}  \| u\|_{S(T)}^{\rho}.
 \end{align}

 Now we set
 \begin{equation*}
     \Phi (u)(t)
     :=U_{\A}(t)\phi +i\int^t_0 U_{\A}(t-\tau) |u(\tau) |^{\rho-1} u(\tau) d\tau.
 \end{equation*}
 Then there exists $C_0>0$ such that
 \begin{equation*}
     \|\Phi u \|_{S(T)}
     \le C_0(\|\phi \|_{\HHSP}+T^{\frac{(\rho-1)(s-s_c)}{\A}} \|u\|_{S(T)}^{\rho} ).
 \end{equation*}
We define
\begin{equation}
  \quad Z(T) \triangleq 
  \{ u\in S(T)\, |\,  \|u\|_{S(T)} \le 2C_0\|\phi \|_{\HHSP}\}
\end{equation}
and 
\begin{equation*}
    T:=K \|\phi \|_{\HHSP}^{-\frac{\A}{s-s_c}},
\end{equation*}
where $K>0$ is a constant.  We assume that
\begin{equation}
    K<(2C_0)^{\frac{\A\rho}{(\rho-1)(s-s_c)}}.\label{Fixedpointcond}
\end{equation}
Then we have
\begin{align*}
    \|\Phi u \|_S
    &\le (C_0+C_0^{\rho+1}2^{\rho} K^{\frac{(\rho-1)(s-s_c)}{\A}} )\|\phi\|_{\HHSP} \\
    &\le 2C_0 \|\phi\|_{\HHSP}. 
\end{align*}
This implies that $\Phi : Z(T)\to Z(T)$ is well defined under the assumption \eqref{Fixedpointcond}.  

Next we estimate the difference $\Phi u_1 -\Phi u_2$ for $u_1,u_2 \in Z(T)$   Arguing as above, we get
\begin{align*}
    \|\Phi u_1 -\Phi u_2 \|_{S(T)}
    &\le 
    C \left\| |D_x|^{s_0-\delta}
    \left( G(u_1) -G(u_2) \right)\right\|_{L^{r_2^{'} }_x L^{q_2^{'}}_t (I_T)   }\\
    &\le  C T^{\frac{(\rho-1)(s-s_c)}{\A}}\left\| |D_x|^{s_0}
    \left( G(u_1) -G(u_2) \right)\right\|_{L^{\tilde{r} }_x L^{\tilde{q}}_t (I_T)   }
 \end{align*}
where $G(u)=|u|^{\rho-1}u$.  We write
\begin{equation*}
    G(u_1)-G(u_2) =
    (u_1-u_2) \int^1_0 G'(\theta u+(1-\theta) u) d
    \theta.
\end{equation*}
Then applying Lemma \ref{Lipe} as above, we have 
\begin{align*}
  &\left\| |D_x|^{s_0}
    \left( G(u_1) -G(u_2) \right)\right\|_{L^{\tilde{r} }_x L^{\tilde{q}}_t (I_T)   } \\
    &\le \||D_x|^{s_0}(u_1 -u_2)\|_{L^{r_1 }_x L^{q_1}_t (I_T) } 
    \int^1_0 \|G'(\theta u_1 +(1-\theta) u_2)  \|_{L^{\frac{r_0}{\rho-1}}_x L^{\frac{q_0}{\rho-1} }_t (I_T)} d  \theta \\
    &+ \| u_1 -u_2\|_{L^{r_0}_x L^{q_0 }_t (I_T)}  \int^1_0
   \| |D_x|^{s_0}
    \{ G'(\theta u_1 +(1-\theta) u_2) \}
    \|_{ L^{r_3 }_x L^{q_3}_t (I_T) }d\theta \\
    &:= I_1+I_2.
\end{align*}
Note that $G' \in {\rm Lip}\,(\rho-1)$ and by definition of the 
Lipschitz norm $G'(z)/|z|^{\rho-1} \le \|G'\|_{{\rm Lip}\, (\rho-1)}$.  Thus
\begin{align*}
    I_1 &\le 
     \||D_x|^{s_0} (u_1 -u_2)\|_{L^{r_1 }_x L^{q_1}_t (I_T) } 
     \|G'\|_{{\rm Lip}\, (\rho-1)}
    \int^1_0 \||\theta u_1 +(1-\theta) u_2|^{\rho-1} \|_{L^{\frac{r_0}{\rho-1}}_x L^{\frac{q_0}{\rho-1} }_t (I_T)} d  \theta \\
    &\le \||D_x|^{s_0}(u_1 -u_2)\|_{L^{r_1 }_x L^{q_1}_t (I_T) } 
     \|G'\|_{{\rm Lip}\, (\rho-1)}
    \int^1_0 (\|u_1 \|_{L^{r_0}_x L^{q_0 }_t (\R)}+ \|u_2\|_{L^{r_0}_x L^{q_0 }_t (I_T)} )^{\rho-1} d  \theta \\
    &\le C(\|u_1\|_{S(T)} +\|u_2\|_{S(T)})^{\rho-1} \|u_1-u_2\|_S.
\end{align*}
For $I_2$ observe that by Lemma \ref{Leibniz} we have
\begin{align*}
    & \| |D_x|^{s_0}
    \{ G'(\theta u_1 +(1-\theta) u_2) \}
    \|_{ L^{r_3 }_x L^{q_3}_t (I_T) }\\
    &\le
    C \|G'\|_{{\rm Lip}\, (\rho-1)}
    \| \theta u_1 +(1-\theta)u_2 \|_{L^{r_0}_x L^{q_0 }_t (\R)}^{\rho-2} \times \| |D_x|^{s_0} (
    \theta u_1 +(1-\theta)u_2)  \|_{L^{r_1 }_x L^{q_1}_t (\R)  } \\
    &\le C (\|u_1\|_{S(T)} +\|u_2\|_{S(T)})^{\rho-1}.
\end{align*}
Therefore, we have
\begin{equation*}
    I_2 \le C(\|u_1\|_{S(T)} +\|u_2\|_{S(T)})^{\rho-1} \|u_1-u_2\|_{S(T)}.
\end{equation*}
Consequently, we see that there exists $C_1>0$ such that
\begin{equation}
  \|\Phi u_1 -\Phi u_2\|_{S(T)}
    \le C_1T^{\frac{(\rho-1)(s-s_c)}{\A}} (\|u_1\|_{S(T)} +\|u_2\|_{S(T)})^{\rho-1} \|u_1-u_2\|_{S(T)}.  \label{difference} 
\end{equation}
By \eqref{difference} we have
\begin{equation*}
    \|\Phi u_1 -\Phi u_2\|_{S(T)}
    \le 2^{\rho-1} C_0^{\rho-1} C_1 K^{\frac{(\rho-1)(s-s_c)}{\A}} \|u_1-u_2\|_{S(T)}
\end{equation*}
for all $u_1,u_2 \in Z(T)$.  Thus $\Phi:Z(T)\to Z(T)$ is a contraction mapping
if we assume moreover that
\begin{equation*}
    K<2^{-\rho}C_0^{1-\rho}C_1.
\end{equation*}
Therefore, by the fixed point theorem, we have 
a solution $u\in Z(T)$ of \eqref{FNLS}.  The uniqueness in the space $S(T)$ follows from \eqref{difference} via a standard argument.  This completes the proof of (i).  The assertion (ii) follows immediately from (i).  Indeed, if we take $T_M:=KM^{-\frac{\A}{s-s_c}}$ for $M>0$, then for any $\phi \in \dot{\mathcal{B}}_M$ there exists a unique solution $u \in S(T_M)$ of \eqref{FNLS} since $T_M\le K\|\phi \|_{\HHSP}^{-\frac{\A}{s-s_c}}$.  Now let $\phi_1,\phi_2 \in \dot{\mathcal{B}}_M$ and let $u_1,u_2 \in S(T_M)$ be solutions with $u_1(0)=\phi_1,\,u_2(0)=\phi_2$.  Then arguing as in the proof of \eqref{difference}, we have
\begin{align*}
  \|u_1-u_2\|_{S(T_M)}
    &\le
    C\|\phi_1-\phi_2\|_{\HHSP}\\
    &+C_1T_M^{\frac{(\rho-1)(s-s_c)}{\A}} (\|u_1\|_{S(T_M)} +\|u_2\|_{S(T_M)})^{\rho-1} \|u_1-u_2\|_{S(T_M)}. 
\end{align*}
The second term of the right-hand side is smaller than
\begin{equation*}
    4^{\rho-1} C_0^{\rho-1}C_1 K^{\frac{(\rho-1)(s-s_c) }{ \A}} \|u_1-u_2\|_{S(T_M)}
\end{equation*}
since $\|u_j\|_{S(T_M)} \le 2C_0 M$.  Thus, this term can be bounded by
\begin{equation*}
\frac{1}{2} \|u_1-u_2\|_{S(T_M)}
\end{equation*}
if we take $K>0$ so small that
\begin{equation}
    4^{\rho-1} C_0^{\rho-1} C_1 K^{\frac{(\rho-1)(s-s_c)}{\A}}<\frac{1}{2}.
    \label{Kcondition}
\end{equation}
 Consequently, we have 
\begin{equation*}
    \|u_1-u_2\|_{S(T_M)} \le 2C\|\phi_1-\phi_2\|_{\HHSP}.
\end{equation*}

This proves \eqref{Cdependenceprop}.  
  \qed \\

\noindent\textit{\textbf{Proof of Theorem \ref{LWPsubcriticalp}} }.
The existence and uniqueness part of the theorem is due to Proposition \ref{KeyEUthm}.  It remains to prove the persistence property, \eqref{Strregularity1}, and the continuous dependence on data of the solution.  We first show (\ref{Strregularity1}).  Let $(q,r,\sigma)$ be $(\A,p)$-acceptable.  We estimate the right hand side of the corresponding integral equation
\begin{equation}
    u(t)=U_{\A}(t)\phi +i\int^t_0 U_{\A}(t-\tau) |u(\tau)|^{\rho-1} u(\tau) d\tau.
\end{equation}
For the linear part, we have
\begin{equation*}
\left\| |D_x|^{s+\sigma} U_{\A}(t) \phi \right\|_{L^{r}_x L^q_t (I_T)} \le C\|\phi \|_{\HHSP}.
\end{equation*}
Noting the $(\A,p')$-acceptablity of $(q_2,r_2,s-s_0+\delta)$ and applying \eqref{key3}, the nonlinear term can be estimated as
\begin{align*}
\left\|\int^t_0 |D_x|^{s+\sigma} U_{\A}(t-\tau) |u(\tau)|^{\rho-1}u (\tau) d\tau\right\|_{L^{r}_x L^{q}_t (I_T)}&\le
 C \left\| |D_x|^s |D_x|^{-(s-s_0+\delta)}|u|^{\rho-1} u  \right\|_{L^{r_2^{'} }_x L^{q_2^{'}}_t (\R)   }.
\end{align*}
We have already seen that the right hand side is bounded by above by
\begin{equation*}
    C T^{\frac{(\rho-1)(s-s_c)}{\A}}  \| u\|_{S(T)}^{\rho}<\infty.
\end{equation*}
Thus we have
\begin{equation*}
    |D_x|^{s+\sigma} u \in L_x^{r}L^{q}_t(I_T).
\end{equation*}
Similarly, we can also show the persistence property of the solution.  We estimate the $\HHSP$-norm 
of the integral equation.  The linear part can be estimated by the trivial identity $\|U_{\A}(t) \phi\|_{\HHSP}=\|\phi \|_{\HHSP},\,\forall t \in \R$.  For the nonlinear part, we apply (\ref{Key2}) to obtain
\begin{align*}
\left\|\int^t_0  U_{\A}(t-\tau) |u(\tau)|^{\rho-1}u (\tau) d\tau\right\|_{L_t^{\infty}(I_T ; \HHSP)}&\le
 C \left\| |D_x|^s |D_x|^{-(s-s_0+\delta)}|u|^{\rho-1} u  \right\|_{L^{r_2^{'} }_x L^{q_2^{'}}_t (\R)   }.
\end{align*}
Arguing as in the proof of (\ref{Strregularity1}), we see that the norm on the left-hand side is finite.  This proves the persistence property of the solution.
This completes the proof of Theorem \ref{LWPsubcriticalp}.

Finally, we prove the continuous dependence on data.  Let $M>0$ and $\phi_1,\phi_2 \in \dot{\mathcal{B}_M}$ and let $u_1,u_2 \in S(T_M)$ be the solution with $u_j(0)=\phi_j,j=1,2$.  Arguing similarly as above we have
\begin{align*}
    \|u_1(t)-u_2(t)\|_{\HHSP}
   & \le \|\phi_1 -\phi_2\|_{\HHSP}+\left\| \int^t_0 U_{\A}(t-\tau) [|u_1|^{\rho-1}u_1 -|u_2|^{\rho-1}u_2 ] d\tau\right\|_{\HHSP} \\
   \le & \|\phi_1 -\phi_2\|_{\HHSP}+
   C_1T_M^{\frac{(\rho-1)(s-s_c) }{\A }} (\|u_1\|_{S(T_M)}+\|u_2\|_{S(T_M)})^{\rho-1} \|u_1-u_2\|_{S(T_M)} \\
    \le & \|\phi_1 -\phi_2\|_{\HHSP} +\frac{1}{2} \|u_1-u_2\|_{S(T_M)},
\end{align*}
where we have used the assumption \eqref{Kcondition} in the last inequality.  Consequently, we get
\begin{equation*}
    \sup_{t\in [0,T_M]}\|u_1(t)-u_2(t)\|_{\HHSP} \le (C+1) \|\phi_1-\phi_2\|_{\HHSP} 
\end{equation*}
by \eqref{Cdependenceprop}.
\bigskip

\subsection{The Critical Case: Proof of Theorem \ref{SmalldataGWP}}. The small data global well-posedness results for $s=s_c$ can be proved in a similar manner.  We set
\begin{equation*}
S:=\{ u|\, u \in L^{r_0}_xL_t^{q_0}(\R),\quad |D_x|^{s_0} u\in L_x^{r_1}L_t^{q_1} (\R)\}
\end{equation*}
with
\begin{equation*}
\|u\|_{S} := \|u \|_{L^{r_0}_xL_t^{q_0}(\R)} +\| |D_x|^{s_0}u \|_{  L_x^{r_1}L_t^{q_1} (\R) }.
\end{equation*}
Then, arguing as in the proof of Theorem \ref{LWPsubcriticalp}, we get
\begin{align}
\|\Phi u \|_{S}
     &\le C_0(\|\phi \|_{\HHSP}+\|u\|_{S}^{\rho} ),\\
       \|\Phi u_1 -\Phi u_2\|_{S}
    &\le C_1 (\|u_1\|_{S} +\|u_2\|_{S})^{\rho-1} \|u_1-u_2\|_{S}.  \label{difference2} 
\end{align}
Let $\epsilon >0$ and assume $\|\phi \|_{\HHSCP} <\epsilon$.  We set $Z_{\epsilon}:=\{u\,|\, \|u\|_S \le 2C_0 \epsilon\}.$  Then, the above estimates give us
\begin{align}
\|\Phi u \|_{S}
     &\le C_0 \epsilon +C_0 2^{\rho}\epsilon^{\rho},\\
       \|\Phi u_1 -\Phi u_2\|_{S}
    &\le 4^{\rho-1} C_0^{\rho-1} C_1 \epsilon^{\rho-1} \|u_1-u_2\|_{S}  \label{difference3} 
\end{align}
for all $u,u_1,u_2 \in Z_{\epsilon}$.
Thus, $\Phi :Z_{\epsilon}\to Z_{\epsilon}$ is well defined and is a contraction mapping if we take $\epsilon$ sufficiently small so that
\begin{equation*}
  \epsilon <\min \left(2^{\frac{\rho}{\rho-1}} C_0^{-1},\, 2^{ \frac{1-2\rho }{\rho-1} } C_0^{-1} C_1^{-\frac{1}{\rho-1}} \right).
\end{equation*}

By the fixed point theorem we get a solution
$u \in Z_{\epsilon}$ of the Cauchy problem.  The other properties of the Theorem can be proved in the same manner.
\qed
\bigskip
\section{Well-posedness in  Inhomogeneous Spaces $\HSP$}\label{proofinhomo}
\subsection{Proof of Corollaries \ref{LWPinhomogeneous} and \ref{SmalldataGWPinhomo}} 

In this section, we prove the well-posedness results in the inhomogeneous Fourier--Sobolev spaces.  Recall first that $q_0,r_0,q_1,r_1$ in Section \ref{proofLWP} depend on $s$.  In this section, we only consider these exponents at $s=s_l$.  More precisely, we set
\begin{align}
\frac{1}{q_0}&=\frac{1}{q_0|_{s=s_l}}=\frac{1}{p}-\frac{2\A}{(2\A-1)(\rho-1)},\\
\frac{1}{r_0}&=\frac{1}{r_0|_{s=s_l}}=\frac{\A}{(2\A-1)(\rho-1)} ,\label{r0defn}\\
\frac{1}{q_1}&=\frac{1}{q_1|_{s=s_l}}=\frac{1}{p}-\frac{2\A}{(2\A-1)(\rho-1)}+\frac{s_0}{2\A-1},  \\
\frac{1}{r_1}&=\frac{1}{r_1|_{s=s_l}}=\frac{\A}{(2\A-1)(\rho-1)}-\frac{s_0}{2\A-1}.\label{r1definition}
\end{align}
Recall that $q_1,r_1$ also depend on $s_0$.  In this sense, we occasionally denote them by $q_1(s_0),r_1(s_0)$ though we mostly write $q_1,r_1$.  We assume that $s_0$ lies in a suitable range.  Let us clarify this before we start the proof.  We set
\begin{equation}
\tilde{s}_0 :=
\begin{cases}
   \max \left(0,\,s_l-\frac{(\A-1)(p-1)}{p} \right)  &\quad \text{if} \quad p\le 2 \\
  \max \left(0,\,s_l-\frac{2\A-p}{2p} \right)
   &\quad \text{if} \quad p\ge 2.
\end{cases}
\end{equation}
Observe that 
\begin{equation*}
 -\frac{(\A-1)(p-1)}{p}\le -\frac{2\A-p}{2p} .
\end{equation*}
if $p\ge 2$ and $\delta=0$ if $s=s_l$.  Thus, $\tilde{s}_0$ is the left-hand side of \eqref{s0range1}, \eqref{s0range2} at $s=s_l$.  Now we assume that $s_0>\tilde{s}_0$ and $s_0$ is sufficiently close to $\tilde{s}_0$.  In particular, $s_0$ satisfies all the assertions of Lemma \ref{exponentlemma} at $s=s_l$.  To prove the corollaries, we need two lemmas.  The first one is the existence of some acceptable triplets and numbers satisfying suitable conditions.

\begin{lem} \label{exponentlem}
There are $q_4,r_4,\sigma_4,q_5,r_5,\sigma_5$ and $\tilde{\delta}>0$ such that $(q_4,r_4,\sigma_4)$ is $(\A,p')$-acceptable, $(q_5,r_5,\sigma_5)$ is $(\A,p)$-acceptable, and
\begin{align}
1-\frac{1}{r_4}+\tilde{\delta}&=\frac{1}{r_5}+\frac{\rho-1}{r_0},\label{tdeltacond}\\ 
-\sigma_4+\tilde{\delta}-s_l&=\sigma_5,\label{tdeltacond2}\\
 \tilde{\delta}&<\frac{1}{r_4},\quad \tilde{\delta}>\sigma_4.\label{tdeltacond3}
\end{align}
\end{lem}

The other key lemma is extra Sobolev-type embeddings for $L_x^rL_t^q$-spaces.

\begin{lem}\label{extrasobolev} Let $T>0$ and let $\A,p,s$ be as in Theorem \ref{subcriticalcondp}. Then:
\begin{enumerate}
\item There exist an $(\A,p)$-acceptable triplet $(q,r,\sigma)$ and $C_T>0$ such that
\begin{equation}
\|u\|_{L^{r_0}_x L_t^{q_0}(I_T)}
\le C_T \left\| |D_x|^{s+\sigma} u \right\|_{L^r_xL_t^q (I_T)} \label{extrasobolev1}
\end{equation}
for any $u$ with $|D_x|^{s+\sigma}u\in L_x^r L_t^q(I_T)$.
\item 
Assume that $s_0>\tilde{s}_0$ and $s_0$ is sufficiently close to $\tilde{s}_0$.  Then there exist an $(\A,p)$-acceptable triplet $(q,r,\sigma)$ and $C_T>0$ such that
\begin{equation}
\||D|^{s_0}u\|_{L^{r_1(s_0)}_x L_t^{q_1(s_0)}(I_T)}
\le C_T \left\| |D_x|^{s+\sigma} u \right\|_{L^r_xL_t^q (I_T)} \label{extrasobolev2}
\end{equation}
for any $u$ with $|D_x|^{s+\sigma}u\in L_x^r L_t^q(I_T)$.
\end{enumerate}
\end{lem}

We postpone the proofs of the two lemmas to the next subsection since they consist mostly of verifying that related exponents satisfy suitable conditions, which is routine but tedious.  We first finish the proof of the corollaries, assuming the lemmas.\\

\noindent\textit{ Proof of Corollary \ref{LWPinhomogeneous}}.  Let $s\in [s_l,s_u)$ and $\phi \in \HSP$.  We want to establish a solution $u\in C([0,T]; \HSP)$ with $u(0)=\phi$.  We first note that 
\begin{equation}
\|u\|_{\HSP} \sim \|u\|_{\Lp}+\|u\|_{\HHSP}
\label{HSPinclusion}
\end{equation}
and $\HSP \subset  \widehat{H^{s_l}_p } \subset \widehat{\dot{H}^{s_l}_p }$.  Thus $\phi \in \widehat{\dot{H}^{s_l}_p } \cap \HHSP  $.  Let $s_0$ be such that $s_0 >\tilde{s}_0$ but it is close enough to $\tilde{s}_0$ so that it satisfies the assumption of Lemma \ref{extrasobolev}.  Then, by Theorem \ref{LWPsubcriticalp}, there is a solution $u_1 \in C([0,T_1]; \widehat{\dot{H}^{s_l}_p })$ of \eqref{FNLS} such that
\begin{equation}
    u_1 \in L^{r_0}_x L_t^{q_0}(I_{T_1}),\quad \text{and}\quad |D_x|^{s_0} u_1 \in L_x^{r_1}L_t^{q_1}(I_{T_1}). \label{corStrregularity-1}
\end{equation}
and
\begin{equation}
    |D_x|^{s_l+\sigma} u_1 \in L^{r}_x L^{q}_t (I_{T_1}) \label{corStrregularity0}
\end{equation}
for any $(\A,p)$-acceptable triplet $(q,r,\sigma)$.  Moreover, by Proposition \ref{KeyEUthm}, the solution is unique in the space
\begin{equation}
S(T)|_{s=s_l} := \{ u\,|\, u \in L_x^{r_0} L_t^{q_0}(I_T),\, |D_x|^{s_0} u \in L_x^{r_1} L_t^{q_1} (I_T)\,\}
\end{equation}
On the other hand, we may also have a solution $u_2 \in C([0,T_2]\,; \HHSP)$ for the same datum $\phi$ such that
\begin{equation}
    |D_x|^{s+\sigma} u \in L^{r}_x L^{q}_t (I_{T_2}) \label{corStrregularity}
\end{equation}
for any $(\A,p)$-acceptable triplet $(q,r,\sigma)$.  The strategy is as follows.  We show the $\Lp$ regularity of $u_1$.  More precisely, we prove
\begin{equation}
\sup_{t\in I_{T_1}} \|u_1 (t) \|_{\Lp}<\infty. \label{Lpregularity}
\end{equation}
Then we prove that $u_1(t)=u_2(t)(:=u(t)),\, t\in [0,\min (T_1,T_2)]$ to conclude the existence of a solution $u \in C([0,\min (T_1,T_2)] \,; \HSP)$ by \eqref{HSPinclusion}.  We prove \eqref{Lpregularity}.  Clearly,\,$\|U_{\A}(t) \phi\|_{\Lp}=\|\phi \|_{\Lp}\le \|\phi \|_{\HSP}$.  We then estimate the Duhamel part of the solution.  By Lemma \ref{exponentlemma}, there are $(\A,p')$-acceptable triplet $(q_4,r_4,\sigma_4)$, $(\A,p)$-acceptable triplet $(q_5,r_5,\sigma_5)$, and $\tilde{\delta}>0$ satisfying \eqref{tdeltacond}--\eqref{tdeltacond3}.

Then, by \eqref{GFSdualprop} , we have
\begin{equation*}
\left\| \int^t_0 U(t-\tau) |u_1(\tau)|^{\rho-1} u_1(\tau) d\tau \right\|_{\Lp} \le C \left\| |D_x|^{-\sigma_4}|u_1|^{\rho-1} u_1 \right\|_{L^{r_4'}_x L_t^{q'_4}(I_T)}
\end{equation*}
We put
\begin{equation}
\frac{1}{r_6}:=1-\frac{1}{r_4}+\tilde{\delta},\quad \frac{1}{q_7}:=1-\frac{1}{q_4}-\frac{1}{q_6}:=1-\frac{1}{q_4}-\left( \frac{1}{q_5}+\frac{\rho-1}{q_0}\right).
\end{equation}
 Then $r_6>1$ since $\tilde{\delta}<r_4^{-1}$.  Moreover, by \eqref{tdeltacond} and the scaling condition \eqref{scaling} for acceptability of $(q_4,r_4,\sigma_4)$ and $(q_5,r_5,\sigma_5)$, we have
\begin{equation*}
\frac{1}{q_7}=2\tilde{\delta}+2-\frac{\rho-1}{p}.
\end{equation*}
It is also easy to see that
$q_7^{-1}\ge 2\tilde{\delta}\ge 0$ since $\rho \le 2p+1$.  Therefore $q_6>1$ and we have

\begin{align*}
 \left\| |D_x|^{-\sigma_4}|u_1|^{\rho-1} u_1 \right\|_{L^{r_4'}_x L_t^{q'_4}(I_{T_1})} &\le  C \left\| |D_x|^{\tilde{\delta}-\sigma_4}|u_1|^{\rho-1} u_1 \right\|_{L^{r_6}_x L_t^{q'_4}(I_{T_1})} \\
&\le CT_1^{2\tilde{\delta}+2-\frac{\rho-1}{p}} C\left\| |D_x|^{\tilde{\delta}-\sigma_4}|u_1|^{\rho-1} u_1 \right\|_{L^{r_6}_x L_t^{q_6}(I_T)},
\end{align*}
by the Sobolev embedding and H\"older's inequality with respect to the time variable. 
Since
\begin{equation*}
\frac{1}{r_6}=\frac{1}{r_5}+\frac{\rho-1}{r_0}\quad \text{and} \quad \tilde{\delta}-\sigma_4>0
\end{equation*}
by \eqref{tdeltacond} and \eqref{tdeltacond3}, we may apply the argument in the proof of Theorem \ref{LWPsubcriticalp} to obtain
\begin{equation*}
\left\| |D_x|^{\tilde{\delta}-\sigma_4}|u_1|^{\rho-1} u_1 \right\|_{L^{r_6}_x L_t^{q_6}(I_T)}   \le C\|u_1 \|_{L_x^{r_0}L_t^{q_0}(I_T)  }^{\rho-1} \||D_x|^{\tilde{\delta}-\sigma_4} u_1 \|_{L_x^{r_5}L_t^{q_5}(I_T)  }.
\end{equation*}
By the assumption $\tilde{\delta}-\sigma_4=s_l+\sigma_5$ and thus we have
\begin{equation}
\left\| \int^t_0 U(t-\tau) |u_1(\tau)|^{\rho-1} u_1(\tau) d\tau \right\|_{\Lp} \le  CT_1^{2\tilde{\delta}+2-\frac{\rho-1}{p}  } \|u_1 \|_{L_x^{r_0}L_t^{q_0}(I_{T_1})  }^{\rho-1} \||D_x|^{s_l+\sigma_5} u_1 \|_{L_x^{r_5}L_t^{q_5}(I_{T_1})  }.
\end{equation}
By \eqref{corStrregularity-1} and \eqref{corStrregularity0}, the norms on the right-hand side are finite.  Thus we have 
\eqref{Lpregularity}.  We now show that $u_1=u_2$.  We choose $s_0$ so that it satisfies the assertion (ii) of Lemma \ref{extrasobolev}.  Then, by \eqref{corStrregularity}, \eqref{extrasobolev1}, and \eqref{extrasobolev2}, we see that $u_2 \in L_x^{r_0}L_t^{q_0}(I_{T_2})$ and $|D_x|^{s_l}u_2 \in L_x^{r_1}L_t^{q_1}(I_{T_2})$, which implies that $u_2\in S(T_2)|_{s=s_l}$.  Thus, by the uniqueness, we have $u_1=u_2$ on $[0,T],$ where $T=\min (T_1,T_2)$. Consequently, we have shown the existence of a solution $u$ for $\phi \in \HSP$ with the persistence property.

This proves Corollary \ref{LWPinhomogeneous}.  Corollary \ref{SmalldataGWPinhomo} can be proved in the same way as Corollary \ref{LWPinhomogeneous}.  In fact, the proof is even simpler, since we do not have to discuss the uniquenss for the case $s=s_l=s_c$.  Note that since the above argument of obtaining the $L^p$ regularity applies for finite time intervals $[0,T],T<\infty$ we can only show that $u \in L^{\infty}_{loc}(\R ;\HSCP)$.  However, we may have $u\in L^{\infty}(\R ; \HSCP)$ if $1/q_7=2\tilde{\delta}+2-\frac{\rho-1}{p}=0$.  This happens only when $p=2,\rho=5$.

\qed

\subsection{Proof of Lemma \ref{exponentlem} and \ref{extrasobolev}}

\noindent\textit{Proof of Lemma \ref{exponentlem}}.  Observe first that $\tilde{\delta}$ is determined by \eqref{tdeltacond}, \eqref{tdeltacond2}.  Indeed, we have
\begin{equation}
    \tilde{\delta}=\frac{s_l}{2\A}=\frac{1}{2(\rho-1)}-\frac{\A-1}{2\A p}.
    \label{tdeltaexpression}.
\end{equation}
To see this, observe first that, by the scaling conditions \eqref{scaling} and \eqref{dgain} 
\begin{equation}
\sigma_4=\frac{\A}{q_4}+\frac{1}{r_4}-\frac{1}{p'}=\A \left(\frac{1}{p'}-\frac{2}{r_4} \right)+\frac{1}{r_4}-\frac{1}{p'}=\frac{\A-1}{p'}-\frac{2\A-1}{r_4} \label{r4expression}
\end{equation}
and 
\begin{equation}
    \sigma_5=\frac{\A}{q_5}+\frac{1}{r_5}-\frac{1}{p}=\frac{\A-1}{p}-\frac{2\A-1}{r_5} \label{r5expression}
\end{equation}
By \eqref{r4expression}, \eqref{r5expression}, and \eqref{tdeltacond2}, we have
\begin{equation}
    \tilde{\delta}=
    \sigma_4+\sigma_5+s_l
    =\A-1-(2\A-1)\left(\frac{1}{r_4}+\frac{1}{r_5}\right)+s_l.\label{tdeltaexpression2}
\end{equation}
On the other hand, we have
\begin{equation}
    \frac{1}{r_4}+\frac{1}{r_5}=\tilde{\delta}+\frac{\A-1}{2\A-1}, \label{r4r5r}
\end{equation}
by \eqref{tdeltacond}.  Thus we get \eqref{tdeltaexpression} from \eqref{tdeltaexpression2} and \eqref{r4r5r}.
Moreover, \eqref{r4r5r} implies that, once we fix either $r_4$ or $r_5$, the other is uniquely determined. The remaining exponents $\sigma_4,q_5,\sigma_5$ are also given by $r_4,r_5$ via the scaling conditions of the acceptability.  Therefore, the proof is reduced to the problem of choosing a suitable $r_4$ or $r_5$.  We will present four choices of $r_4$ or $r_5$ that give acceptable triplets with the desired condition.  We need additional notations for Example $1$ and $2$ below.  We set
\begin{equation*}
R(\A,p,\rho) :=\frac{1}{2(\rho-1)}-\frac{\A-1}{2\A p} -\frac{1}{4}+\frac{\A-1}{2\A-1},
\end{equation*}
and
\begin{align*}
A_1(\A,p):= \{\rho \in I_{\A,p} \,|\, R(\A,p,\rho) \ge 0\,\},\quad 
A_2(\A,p):= \{\rho \in I_{\A,p} \,|\, R(\A,p,\rho) < 0\,\},
\end{align*}
where
\begin{equation*}
    I_{\A,p} :=\biggl]\max\left(\frac{6\A-1}{2\A-1},\,\frac{2\A-1+2\A p}{2\A-1}  \right), \,2p+1\biggr].
\end{equation*}

\noindent\textbf{Example 1}.\, $1/r_4=1/4-\varepsilon$, 
where $\varepsilon>0$ is sufficiently small. \\

\noindent\textbf{Example 2}.\, $1/r_5=\varepsilon$, where $\varepsilon>0$ is sufficiently small.\\

Then we have: 

\noindent\textbf{Claim 1}.
\begin{enumerate}
\item Example 1 gives $(q_4,r_4,\sigma_4), (q_5,r_5,\sigma_5)$ and $\tilde{\delta}$ with the desired properties if $2 \le p \le \min (4,4\A-2)$ and $\rho \in A_1(\A,p)$.

\item 
Example 2 gives $(q_4,r_4,\sigma_4), (q_5,r_5,\sigma_5)$ and $\tilde{\delta}$ with the desired properties if $2 \le p \le \min (4,4\A-2)$ and $\rho \in A_2(\A,p)$.
\end{enumerate}

\textit{Proof of Claim 1}.\, (i)  Let $r_4^{-1}=1/4-\varepsilon$.  Note first that $(q_4,r_4,\sigma_4)$ is $(\A, p')$-acceptable if $p'\le 2$ and $\varepsilon$ is sufficiently small (see Figure \ref{fig:acceptable} below.) We check that $\tilde{\delta}<r_4^{-1}$.  The condition is equivalent to 
\begin{equation}
    \frac{1}{2\A} \left( \frac{\A}{\rho-1} -\frac{\A-1}{p} \right) <\frac{1}{4}-\varepsilon.
\end{equation}
Since $\varepsilon$ can be taken small, it is enough to show that
\begin{equation}
    \frac{1}{2\A} \left( \frac{\A}{\rho-1} -\frac{\A-1}{p} \right) <\frac{1}{4}. \label{note1}
\end{equation}
By the assumption, $\frac{1}{\rho-1} <\frac{2\A-1}{2\A p}$. Thus \eqref{note1} holds if
\begin{equation}
    \frac{1}{2\A} \left(
    \A \cdot \frac{2\A-1}{2\A p} -\frac{\A-1}{p} \right) <\frac{1}{4},
\end{equation}
which is equivalent to
\begin{equation}
    \frac{2\A-1}{2p} -\frac{\A-1}{p}<\frac{\A}{2}.
\end{equation}
The condition can be rewritten as $p>\A^{-1}$, which clearly holds true.  Next we show that
$(q_5,r_5,\sigma_5)$ determined by \eqref{r4r5r} and the scaling conditions is $(\A,p)$-acceptable.  When $p\ge 2$, it is enough to check that $0<r_5^{-1}<1/2p$. (see Figure \ref{fig:acceptable} below.)  By \eqref{r4r5r},
\begin{equation}
    \frac{1}{r_5}=\tilde{\delta}+\frac{\A-1}{2\A-1}-\frac{1}{r_4}=\frac{1}{2(\rho-1)}-\frac{\A-1}{2\A p}+\frac{\A-1}{2\A-1}-\frac{1}{4}+\varepsilon.
\end{equation}
Thus, we see that $r_5^{-1}>0$ since $\rho\in A_1(\A,p)$.  To prove that $r_5^{-1}<1/2p$ it is enough to show that
\begin{equation}
\frac{2\A-1}{4\A p} -\frac{\A-1}{2\A p} -\frac{1}{4} +\frac{\A-1}{2\A-1} <\frac{1}{2p}
\end{equation}
since $\varepsilon$ is small and $\frac{1}{\rho-1}< \frac{2\A-1}{2\A p}$.  This is equivalent to
\begin{equation}
(2\A-3) p <\frac{(2\A-1)^2}{\A}.
\end{equation}
If $\A\le 3/2$, the inequality holds.  When $\A>3/2$, the condition is rewritten as 
\begin{equation}
    p<\frac{(2\A-1)^2}{\A (2\A-3)}.
\end{equation}
Now it is easy to see that
\begin{equation*}
    4<\frac{(2\A-1)^2}{\A (2\A-3)},
\end{equation*}
for all $\A\in [3/2,2]$.  Therefore, $1/r_5<1/2p$ for all $p \in [2,\min (4,4\A-2))$ and $\A \in (1,2]$.


\begin{figure}[htbp]
  \centering

\begin{tikzpicture}[scale=20]
\draw[->] (0,0) -- (0.55,0) node[right] {$\frac{1}{q}$};
\draw[->] (0,0) -- (0,0.30) node[left] {$\frac{1}{r}$};

\draw (0.5,0) node[below] {$\frac12$};
\draw (0.25,0) node[below] {$\frac14$};
\draw (0,0.25) node[left] {$\frac14$};

\draw[dashed] (0,0.25) -- (0.25,0.25);
\draw[dashed] (0.25,0.25) -- (0.5,0);
\draw[thick] (0,1/4) -- (1/2,0);

\draw[dashed]
  (0,1/3) -- (1/6,1/4);
\draw[dashed]
  (0,1/3) -- (2/3,0);
\draw[thick]
  (1/6,1/4) -- (1/3,1/6);

\draw[dashed] (0.25,0.25) -- (0.5,0);

\draw[thick]
  (0,1/12) -- (1/6,0); 
\draw[dashed]
  (-1/36,7/72) --(0,1/12);
\draw[dashed]
  (1/6,0)--(1/5,-1/60);
\draw[thick]
  (0,0.25) -- (0.25,0.125);


\node at (1/2,1/8) {$\ell_p\,(p<2)$};


\node at (1/12,1/14) {$\ell_p\,(p>2)$};

\node[above,font=\footnotesize,xshift=12pt]
  at ({1/6},{1/4}) {$(\frac{1}{p}-\frac{1}{2},\frac{1}{4})$};
  
\filldraw[fill=white,draw=black] ({1/6},{1/4}) circle[radius=0.15pt];

\node[above right,font=\footnotesize,xshift=-3pt]
  at ({1/3},{1/6}) {$(1-\frac{1}{p},\frac{1}{p}-\frac{1}{2})$};

\filldraw[fill=white,draw=black] ({1/3},{1/6}) circle[radius=0.15pt];


\filldraw[fill=white,draw=black] ({0},{1/12}) circle[radius=0.15pt];

\filldraw[fill=white,draw=black] ({1/6},{0}) circle[radius=0.15pt];

\draw (0,1/12) node[left,yshift=-4pt] {$\frac{1}{2p}$};

\draw (1/6,0) node[below,yshift=-2pt] {$\frac{1}{p}$};

\end{tikzpicture}

\caption{The lines $\ell_p: 1/q+2/r=1/p$ of acceptable pairs $(1/q,1/r)$.}
  \label{fig:acceptable}
\end{figure}
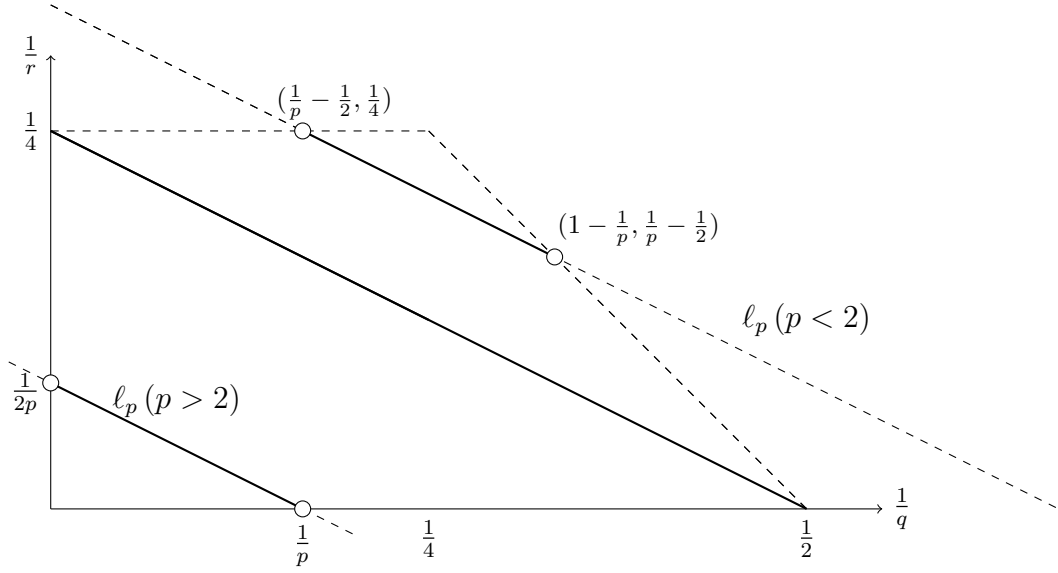


It remains to check that $\tilde{\delta}>\sigma_4$.  By \eqref{r4expression}, we have
\begin{align}
\sigma_4&=\frac{\A-1}{p'}-\frac{2\A-1}{r_4}\\
&= (\A-1)(1-\frac{1}{p}) -(2\A-1)(\frac{1}{4}-\varepsilon) \\
    &=\frac{\A}{2}-\frac{\A-1}{p}-\frac{3}{4}+\varepsilon (2\A-1).
\end{align}
To show $\sigma_4 <\tilde{\delta}$, it is enough to check that
\begin{equation}
    \frac{\A}{2}-\frac{\A-1}{p}-\frac{3}{4} <\frac{1}{2(\rho-1)}-\frac{\A-1}{2\A p}.
\end{equation}
Since $\frac{1}{\rho-1}\ge \frac{1}{2p}$, it suffices to show
\begin{equation*}
\frac{\A}{2}-\frac{\A-1}{p}-\frac{3}{4} <\frac{1}{4p}-\frac{\A-1}{2\A p}.
\end{equation*}
This is equivalent to
\begin{equation*}
    p(2\A-3)< \frac{4\A^2-5\A +2}{\A}.
\end{equation*}
If $\A\le 3/2$, the inequality holds.  When $\A >3/2$, the condition is
\begin{equation*}
    p<\frac{4\A^2-5\A+2}{\A (2\A-3)}.
\end{equation*}
This also holds true since
\begin{equation*}
    4\le\frac{4\A^2-5\A+2}{\A (2\A-3)},
\end{equation*}
if and only if $-1/4\le \A\le 2$.

We prove (ii).  Let $r_5^{-1}=\varepsilon$.  It is clear from Figure \ref{fig:acceptable} above that $(q_5,r_5,\sigma_5)$ is $(\A,p)$-acceptable since $p\ge 2$.  It is also easy to see that $\tilde{\delta}<r_4^{-1}$ by \eqref{r4r5r} since $r_5^{-1}$ is sufficiently small.  We show that $(q_4,r_4,\sigma_4)$ is $(\A,p')$-acceptable.  We need to check that (see Figure \ref{fig:acceptable} )
\begin{equation}
    \frac{1}{p'}-\frac{1}{2}< \frac{1}{r_4} <\frac{1}{4}.\label{r4accep}
\end{equation}
Since 
\begin{equation*}
\frac{1}{r_4}=\tilde{\delta}+\frac{\A-1}{2\A-1}-\varepsilon=\frac{1}{2(\rho-1)}-\frac{\A-1}{2\A p}+
\frac{\A-1}{2\A-1}-\varepsilon
\end{equation*}
and $\rho \le 2p+1$, the  first inequality holds if
\begin{equation*}
    \frac{1}{4p}-\frac{\A-1}{2\A p}+\frac{\A-1}{2\A-1}>\frac{1}{2}-\frac{1}{p}.
\end{equation*}
This is rewritten as
\begin{equation*}
    p<\frac{(3\A+2)(2\A-1)}{2\A}.
\end{equation*}
Now it is easy to see that this inequality holds since the right-hand side is larger than or equal to $4\A-2$ when $\A\le 2$.  The second inequality of \eqref{r4accep} immediately follows from the assumption that $\rho \in A_2(\A,p)$.  Finally, we show that $\sigma_4<\tilde{\delta}$.  By \eqref{r4expression}, we have
\begin{align}
    \sigma_4 &= \frac{\A-1}{p'}-\frac{2\A-1}{r_4}\\
    &=\frac{\A-1}{p'}-(2\A-1)\left( \tilde{\delta}+\frac{\A-1}{2\A-1}-\varepsilon \right) \\
    &=-\frac{\A-1}{p}-(2\A-1)\tilde{\delta}+(2\A-1)\varepsilon.
\end{align}
Thus $\sigma_4<\tilde{\delta}$ is equivalent to
\begin{equation*}
    2\A\tilde{\delta} >-\frac{\A-1}{p}+(2\A-1)\varepsilon.
\end{equation*}
It is clear that this inequality holds since the right-hand side is negative if $\varepsilon$ is sufficiently small.\qed

We need two more examples to prove the lemma for the case $p\le 2$:

\noindent\textbf{Example 3}.\, We set 
\begin{equation*}
    \frac{1}{r_5}=\frac{\A}{(\rho-1)(2\A-1)} -\varepsilon,
\end{equation*}
where $\varepsilon>0$ is sufficiently small. \\

\noindent\textbf{Example 4}.\, We set
\begin{equation*}
    \frac{1}{r_5}=\frac{\A-1}{2\A-1}-\varepsilon,
\end{equation*}
where $\varepsilon>0$ is small. \\

Then we have:

\noindent\textbf{Claim 2}.
\begin{enumerate}
\item Example 3 gives $(q_4,r_4,\sigma_4),\,(q_5,r_5,\sigma_5)$, and $\tilde{\delta}$ with the desired properties if 
\begin{equation*}
    \frac{2\A}{2\A-1}\le p \le 2 \quad \text{and}\quad \rho \ge 1+\frac{\A}{\A-1}.
\end{equation*}

\item 
Example 4 gives $(q_4,r_4,\sigma_4),\,(q_5,r_5,\sigma_5)$, and $\tilde{\delta}$ with the desired properties if \begin{equation*}
    \A \le 3/2,\quad \frac{2\A}{2\A-1}\le p \le 2 \quad \text{and}\quad \rho \le 1+\frac{\A}{\A-1}.
\end{equation*}

\end{enumerate}

Note that when $\A \ge 3/2$, $\rho$ satisfies the assumption of (i) of the above claim since $\rho >\frac{6\A-1}{2\A-1}$.\\

\noindent\textit{Proof of Claim 2}.\, (i) We first show that $\tilde{\delta}<r_4^{-1}$.  Since
\begin{equation*}
    \frac{1}{r_4}=\tilde{\delta}-\frac{1}{r_5}+\frac{\A-1}{2\A-1},
\end{equation*}
this is equivalent to
\begin{equation*}
\frac{\A-1}{2\A-1} >\frac{1}{r_5}.
\end{equation*}
Thus it is enough to check that
\begin{equation*}
    \frac{\A-1}{2\A-1}\ge \frac{\A}{(\rho-1)(2\A-1)},
\end{equation*}
from which we need
\begin{equation*}
    \rho \ge 1+\frac{\A}{\A-1}.
\end{equation*}
Next we show that $1/p-1/2<1/r_5<1/4$ to prove the $(\A,p)$-acceptablity of $(q_5,r_5,\sigma_5)$.  Since $\varepsilon$ is small enough, the first inequality follows if 
\begin{equation*}
\frac{1}{p}-\frac{1}{2} <\frac{\A}{(\rho-1)(2\A-1)}.
\end{equation*}
Since $\rho \le 2p+1$, this holds true if
\begin{equation*}
    p>\frac{3\A-2}{2\A-1}.
\end{equation*}
We have this inequality since
\begin{equation*}
    \frac{4}{3}\ge \frac{3\A-2}{2\A-1} 
\end{equation*}
for $\A\le 2$.  The second inequality holds true if
\begin{equation*}
    \frac{\A}{(\rho-1)(2\A-1)} \le \frac{1}{4}.
\end{equation*}
It is easy to see that this follows from the assumption $\rho >\frac{6\A-1}{2\A-1}$.  Thus $(q_5,r_5,\sigma_5)$ is $(\A,p)$-acceptable.  Next we show that $(q_4,r_4,\sigma_4)$ is $(\A,p')$-acceptable.  Since $p'\ge 2$, it suffices to show that $0<1/r_4<1/2p'$.  The first inequality is obvious since we have already shown that $1/r_4>\tilde{\delta}(\ge 0)$.  We show that $1/r_4<1/2p'$.  We have
\begin{align*}
    \frac{1}{r_4} &=\tilde{\delta}-\frac{1}{r_5}+\frac{\A-1}{2\A-1}=\frac{1}{2(\rho-1)}-\frac{\A-1}{2\A p}-\frac{\A}{(\rho-1)(2\A-1)}+\varepsilon +\frac{\A-1}{2\A-1} \\
    &=\frac{\A-1}{2\A-1}-\frac{\A-1}{2\A p}-\frac{1}{2(\rho-1)(2\A-1)}+\varepsilon.
\end{align*}
Since $\varepsilon$ is small enough and $\rho \le 2p+1$, the inequality holds true if
\begin{equation*}
    \frac{\A-1}{2\A-1}-\frac{\A-1}{2\A p}-\frac{1}{4p(2\A-1)}<\frac{1}{2}-\frac{1}{2p}.
\end{equation*}
This is equivalent to
\begin{equation*}
    p>\frac{3\A-2}{2\A}.
\end{equation*}
We have this inequality since 
\begin{equation*}
    \frac{2\A}{2\A-1}\ge \frac{3\A-2}{2\A}
\end{equation*}
for $1<\A\le 2$.
Therefore, $(q_4,r_4,\sigma_4)$ is $(\A,p')$-acceptable.  Finally, we prove that $\tilde{\delta}>\sigma_4$.  By \eqref{r4expression}, we have
\begin{align*}
\sigma_4&=\frac{\A-1}{p'}-\frac{2\A-1}{r_4}=\frac{\A-1}{p'}-(2\A-1)\left( \tilde{\delta}+\frac{\A-1}{2\A-1}-\frac{1}{r_5}\right)\\
&=-(2\A-1)\tilde{\delta}-\frac{\A-1}{p}+\frac{2\A-1}{r_5}.
\end{align*}
Thus the inequality is equivalent to
\begin{equation*}
2\A\tilde{\delta} >\frac{2\A-1}{r_5}-\frac{\A-1}{p}.
\end{equation*}
It is clear that this can be rewritten as 
\begin{equation*}
    \frac{1}{r_5}<\frac{\A}{(\rho-1)(2\A-1)},
\end{equation*}
which is obvious by the definition of $r_5$.

(ii)\, We consider Example $4$.  We first note that $\tilde{\delta}<r_4^{-1}$, which is obvious because
\begin{equation*}
    \frac{1}{r_4}=\tilde{\delta}-\frac{1}{r_5}+\frac{\A-1}{2\A-1}=\tilde{\delta}+\varepsilon.
\end{equation*}
We show that $(q_5,r_5,\sigma_5)$ is $(\A,p)$-acceptable by showing that $1/p-1/2<1/r_5<1/4$.  The first inequality holds if
\begin{equation*}
    \frac{1}{p}-\frac{1}{2}<\frac{\A-1}{2\A-1},
\end{equation*}
from which we need
\begin{equation*}
    p>\frac{2(2\A-1)}{4\A-3}.
\end{equation*}
$p$ satisfies this condition since
\begin{equation*}
    \frac{2\A}{2\A-1}\ge \frac{2(2\A-1)}{4\A-3}
\end{equation*}
when $\A\ge 1$.  For the second inequality it is enough to show that
\begin{equation*}
    \frac{\A-1}{2\A-1} \le \frac{1}{4}.
\end{equation*}
Obviously, this holds true if and only if $\A \le 3/2$.  Consequently, $(q_5,r_5,\sigma_5)$ is $(\A,p)$-acceptable when $\A\le 3/2$.  Next we prove that $(q_4,r_4,\sigma_4)$ is $(\A,p')$-acceptable.  Since $r_4^{-1}>\tilde{\delta}\ge 0$, it is enough to show that $1/r_4<1/2p'$.  By the definition of $r_5,\tilde{\delta}$ and \eqref{r4r5r}, we have 
\begin{equation*}
    \frac{1}{r_4}=\frac{1}{2(\rho-1)}-\frac{\A-1}{2\A p} +\varepsilon.
\end{equation*}
Since $\varepsilon$ is small and $\rho>\frac{6\A-1}{2\A-1}$, it suffices to show that
\begin{equation*}
\frac{2\A-1}{8\A}-\frac{\A-1}{2\A p} <\frac{1}{2}-\frac{1}{2p}.
\end{equation*}
This is equivalent to
\begin{equation*}
p>\frac{4}{2\A+1},
\end{equation*}
which holds true since the right-hand side is smaller than or equal to $4/3$ when $\A \ge 1$.  
It remains to show that $\tilde{\delta}>\sigma_4$.  Arguing as in the proof of (i), this amounts to proving
\begin{equation*}
2\A\tilde{\delta} >\frac{2\A-1}{r_5}-\frac{\A-1}{p}.
\end{equation*}
This inequality is equivalent to
\begin{equation*}
    \frac{\A}{\rho-1}>\A-1-(2\A-1)\varepsilon.
\end{equation*}
It is easy to see that this holds true if $\rho \le 1+\A/(\A-1)$.

Now it is clear from the above examples provide $q_4,r_4,\sigma_4,q_5,r_5,\sigma_5$ and $\tilde{\delta}$ with the desired properties for any $p,\A,\rho$ satisfying the assumption s of Theorem \ref{LWPsubcriticalp}.

\qed

\noindent\textit{Proof of Lemma \ref{extrasobolev}}.
To prove (i) we want to apply the Sobolev type inequality \eqref{Sobolev} to obtain
\begin{equation}
\|u\|_{L^{r_0}_x L_t^{q_0}(I_T)}
\le C\left\| |D_x|^{s+\sigma} u \right\|_{L^r_xL_t^{q_0} (I_T)}
\end{equation}
 for which we need
\begin{equation}
    \frac{1}{r}-\frac{1}{r_0} =s+\sigma. \label{sobolevscaling}
\end{equation}
Once the above estimate is obtained, the desired inequality immediately follows from H\"older's inequality with respect to the time variable.
We first determine $q$, and then define $r$ and $\sigma$ according to the scaling conditions
\begin{equation}
    \frac{1}{q}+\frac{2}{r}=\frac{1}{p},\label{sobolevscaling1}
    \end{equation}
    and
    \begin{equation}
        \sigma =\frac{\A}{q}+\frac{1}{r}-\frac{1}{p}. \label{sobolevscaling2}
\end{equation}
By \eqref{sobolevscaling}, \eqref{sobolevscaling2}, and \eqref{r0defn}, we have
\begin{equation*}
    \frac{1}{r}-\frac{\A}{(\rho-1)(2\A-1)}=s+\frac{\A}{q}+\frac{1}{r}-\frac{1}{p},
\end{equation*}
from which we obtain
\begin{equation}
\frac{\A}{q}=\frac{1}{p}-s-\frac{\A}{(\rho-1)(2\A-1)}.\label{qexpression}
\end{equation}
We first prove that $s+\sigma\ge 0$ so that Sobolev's embedding can apply.  We have
\begin{equation*}
    s+\sigma=\frac{1}{r}-\frac{1}{r_0}=\frac{1}{r}-\frac{\A}{(\rho-1)(2\A-1)}.
\end{equation*}
Since $1/r=1/2(1/p-1/q)$, it suffices to show that
\begin{equation*}
\frac{1}{q} \le \frac{1}{p}-\frac{2\A}{(\rho-1)(2\A-1)}.
\end{equation*}
By \eqref{qexpression}, this is equivalent to
\begin{equation*}
    \frac{1}{\A p}-\frac{s}{\A}-\frac{1}{(\rho-1)(2\A-1)} \le \frac{1}{p}-\frac{2\A}{(\rho-1)(2\A-1)}.
\end{equation*}
Now it is easy to see that this can be rewritten as
\begin{equation*}
s\ge \frac{\A}{\rho-1}-\frac{\A-1}{p}=s_l,
\end{equation*}
which is satisfied by the assumption.  It remains to show that $(q,r,\sigma)$ is $(\A,p)$-acceptable.  Since $q,r,\sigma$ satisfy \eqref{sobolevscaling1}, \eqref{sobolevscaling2}, it is enough to verify that $q$ lies in a suitable range.  We first assume that $p\ge 2$.  Then $(q,r,\sigma)$ is $(\A,p)$-acceptable if $0<q^{-1}<1/p$ (see Figure \ref{fig:acceptable} above.)  We first check that $q^{-1}>0$.  By \eqref{qexpression}, the condition is equivalent to
\begin{equation*}
    s<\frac{1}{p}-\frac{\A}{(2\A-1)(\rho-1)}.
\end{equation*}
Since $\frac{1}{\rho-1}\le \frac{2\A-1}{2\A p}$, the inequality holds if $s<1/2p$.  It is clear that this condition is satisfied since $s<s_u\le 1/2p$.  Thus, we see that $0<q^{-1}$.  We can skip the proof of the inequality $q^{-1}<p^{-1}$ here.  For details, see the proof of the assertion (ii) below. 

Next we assume $p\le 2$.  Then $(q,r,\sigma)$ is $(\A,p)$-acceptable if
\begin{equation}
    \frac{1}{p}-\frac{1}{2}<q<1-\frac{1}{p}. \label{acceptablerange1}
\end{equation}
The first inequality is equivalent to
\begin{equation*}
    \frac{1}{p}-\frac{1}{2} <\frac{1}{\A p}-\frac{s}{\A}-\frac{1}{(\rho-1)(2\A -1)}.
\end{equation*}
Since $\rho-1 >\frac{4\A}{2\A-1}$, the equality holds true if
\begin{equation*}
    s<\frac{4-4\A+2\A p-p}{4p}.
\end{equation*}
Thus the condition is fullfilled since the right-hand side is larger than or equal to $s_u$.  Thus the first inequality of \eqref{acceptablerange1} is proved.  As in the case of $p\ge 2$, we can skip the proof of the second inequality.  Consequently, $(q,r,\sigma)$ is $(\A,p)$-acceptable and the assertion (i) is proved.

The proof of (ii) is similar.  We use the Sobolev type embedding with
\begin{equation}
    \frac{1}{r}-\frac{1}{r_1(s_0)}=s+\sigma-s_0
    \label{sobolevscaling3}
\end{equation}
to obtain
\begin{equation}
\left\| |D_x|^{s_0}u \right\|_{L_x^{r_1(s_0)}L_t^{q_1(s_0)}(I_T)} \le C \left\| |D_x|^{s+\sigma} u  \right\|_{L^r_x L_t^{q_1(s_0)} (I_T)} 
\end{equation}
and conclude the proof via the H\"older's inequality.  By \eqref{r1definition}, \eqref{sobolevscaling2}, and \eqref{sobolevscaling3}, we have
\begin{align*}
    \frac{\A}{q}&=\sigma-\frac{1}{r}+\frac{1}{p} =s_0-s-\frac{1}{r_1(s_l)}+\frac{1}{p} \\
    &=s_0-s+\frac{1}{p}-\frac{\A}{(\rho-1)(2\A-1)}+\frac{s_0}{2\A-1}\\
    &=-s+\frac{1}{p}-\frac{\A}{(\rho-1)(2\A-1)}+\frac{2\A }{2\A-1}s_0.
\end{align*}
In particular,
\begin{equation}
    \frac{1}{q}=-\frac{s}{\A}+\frac{1}{\A p}-\frac{1}{(\rho-1)(2\A-1)}+\frac{2s_0 }{2\A-1}. \label{qexpression2}
\end{equation}
We then define $r,\sigma$ by \eqref{sobolevscaling1} and \eqref{sobolevscaling2}.  We first show that $s+\sigma-s_0\ge 0$.  Noting that
\begin{equation*}
    \sigma=\frac{2\A-1}{2q}-\frac{1}{2p}
\end{equation*}
by \eqref{sobolevscaling1} and \eqref{sobolevscaling2}, the inequality is equivalent to
\begin{equation}
    s+\frac{2\A-1}{2q}-\frac{1}{2p} -s_0\ge 0.
\end{equation}
By \eqref{qexpression2}, the left-hand side is equal to
\begin{align*}
s+\frac{2\A-1}{2}\biggl[-\frac{s}{\A}+\frac{1}{\A p}-\frac{1}{(\rho-1)(2\A-1)}&+\frac{2s_0 }{2\A-1}\biggr]-\frac{1}{2p}-s_0 \\
&=\frac{s}{2\A}+\frac{2\A-1}{2\A p}-\frac{1}{2(\rho-1)}-\frac{1}{2p} \\
&=\frac{1}{2\A}(s-s_l).
\end{align*}
Clearly, the right-hand side is non-negative since $s\ge s_l$.  Hence $s+\sigma-s_0\ge 0$.  It remains to prove that $(q,r,\sigma)$ is $(\A,p)$-acceptable.  As in the proof of (i), we show that
\begin{align}
0<\frac{1}{q}&<\frac{1}{p}, \quad  \text{when} \quad p\ge 2 \label{sobolevqrange1} \\
\frac{1}{p}-\frac{1}{2} <\frac{1}{q}& <1-\frac{1}{p}, \quad \text{when} \quad p \le 2.\label{sobolevqrange2}
\end{align}
We observe that
\begin{align*}
    \frac{1}{q}&=\left[\frac{1}{\A p} -\frac{s}{\A} -\frac{1}{(\rho-1)(2\A-1)}  \right] +\frac{2s_0}{2\A-1}\\
    &\ge \frac{1}{\A p} -\frac{s}{\A} -\frac{1}{(\rho-1)(2\A-1)}.
\end{align*}
The right-hand side is ``$q^{-1}$'' in the proof of (i).  Thus, we do not need to check the first inequalities of \eqref{sobolevqrange1} and \eqref{sobolevqrange2}.  Similarly, the inequalities ``$q^{-1}<p^{-1}$'' for $p\ge 2$ and ``$q^{-1}<1-1/p$'' for $p\le 2$ in the proof of (i) follow immediately once we prove the second inequalities of \eqref{sobolevqrange1} and \eqref{sobolevqrange2} here.  We prove the second inequalities of \eqref{sobolevqrange1} and \eqref{sobolevqrange2}.  We first assume that $p\ge 2$.  We have to show that
\begin{equation}
\frac{1}{q}=\frac{1}{\A p}-\frac{s}{\A}-\frac{1}{(\rho-1)(2\A-1)} +\frac{2s_0}{2\A-1} <\frac{1}{p}. \label{qinequality}
\end{equation}
This is equivalent to
\begin{equation}
    s_0<\frac{2\A-1}{2p}-\frac{2\A-1}{2\A p}+\frac{1}{2(\rho-1)}+\frac{2\A-1}{2\A}s.
    \label{qinequality2}
\end{equation}
Now we recall the definition of $\tilde{s}_0$.  We need to take $s_0$ so that $s_0>\tilde{s}_0$ and it is close enough to $\tilde{s}_0$.  In view of \eqref{qinequality2} we can take an $s_0$ so that it lies in the required range and $q$ satisfies \eqref{qinequality} if
\begin{equation}
    \tilde{s}_0 <\frac{2\A-1}{2p}-\frac{2\A-1}{2\A p}+\frac{1}{2(\rho-1)}+\frac{2\A-1}{2\A}s.
\end{equation}
We check this inequality.  It is clear that its right-hand side is positive.  So it remains to show that
\begin{equation*}
    s_l-\frac{2\A-p}{2p} <\frac{2\A-1}{2p}-\frac{2\A-1}{2\A p}+\frac{1}{2(\rho-1)}+\frac{2\A-1}{2\A}s.
\end{equation*}
Since $s\ge s_l$, it is enough to prove the inequality with $s=s_l$; that is,
\begin{equation*}
    \frac{s_l}{2\A} <\frac{2\A-p}{2p}+\frac{2\A-1}{2p}-\frac{2\A-1}{2\A p} +\frac{1}{2(\rho-1)}.
\end{equation*}
By the definition of $s_l$ and a simple computation, we see that this is equivalent to $p<4\A-2$.  Hence we get the desired inequality.
  It remains to show $q^{-1}<1-1/p$ when $p\le 2$.  We have to show that 
\begin{equation*}
\frac{1}{\A p}-\frac{s}{\A}-\frac{1}{(\rho-1)(2\A-1)} +\frac{2s_0}{2\A-1} <1-\frac{1}{p},
\end{equation*}
which is equivalent to 
\begin{equation*}
s_0 <\frac{2\A-1}{2\A}s -\frac{2\A-1}{2\A p}+\frac{1}{2(\rho-1)}+\frac{2\A-1}{2}-\frac{2\A-1}{2p}.
\end{equation*}
As in the case of $p\ge 2$, the desired $s_0$ exists if
\begin{equation*}
   \tilde{s}_0 <\frac{2\A-1}{2\A}s -\frac{2\A-1}{2\A p}+\frac{1}{2(\rho-1)}+\frac{2\A-1}{2}-\frac{2\A-1}{2p}.
\end{equation*}
We first check that the right-hand side of the above inequality is positive.  Since $s\ge s_l$ and $\rho \le 2p+1$
\begin{align*}
    \frac{2\A-1}{2\A}s -\frac{2\A-1}{2\A p}+&\frac{1}{2(\rho-1)}+\frac{2\A-1}{2}-\frac{2\A-1}{2p} \\
    \ge& \frac{2\A-1}{2\A}\left(\frac{\A}{\rho-1}-\frac{\A-1}{p} \right) -\frac{2\A-1}{2\A p}+\frac{1}{2(\rho-1)}+\frac{2\A-1}{2}-\frac{2\A-1}{2p} \\
    =&\frac{\A}{\rho-1}-\frac{2\A-1}{p}+\frac{2\A-1}{2} \\
    \ge & \frac{\A}{2p}-\frac{2\A-1}{p}+\frac{2\A-1}{2}.
\end{align*}
The last term is positive if and only if
\begin{equation*}
    p>\frac{3\A-2}{2\A-1},
\end{equation*}
which is satisfied since when $\A\le 2$ and  $p>\frac{2\A}{2\A-1}$.  It remains to prove
\begin{equation}
    s_l-\frac{(\A-1)(p-1)}{p}<
 \frac{2\A-1}{2\A}s -\frac{2\A-1}{2\A p}+\frac{1}{2(\rho-1)}+\frac{2\A-1}{2}-\frac{2\A-1}{2p}.
\end{equation}
It suffices to show  the inequality with $s=s_l$, which is rewritten as 
\begin{equation*}
    \frac{s_l}{2\A}<\frac{(\A-1)(p-1)}{p}-\frac{2\A-1}{2\A p}+\frac{1}{2(\rho-1)}+\frac{2\A-1}{2}-\frac{2\A-1}{2p},
\end{equation*}
which is equivalent to
\begin{equation*}
    p>\frac{4\A-2}{4\A-3}.
\end{equation*}
This inequality holds since
\begin{equation*}
   \frac{2\A}{2\A-1}>\frac{4\A-2}{4\A-3}.
\end{equation*}
\qed

\bibliographystyle{siam}
\bibliography{main}

\end{document}